\documentclass[12pt]{amsart}
\usepackage{a4wide,enumerate,xcolor,graphicx}
\usepackage{amsmath}
\usepackage{float}
\allowdisplaybreaks

\let\pa\partial
\let\na\nabla
\let\eps\varepsilon
\newcommand{\N}{{\mathbb N}}
\newcommand{\R}{{\mathbb R}}

\newcommand{\diver}{\operatorname{div}}

\newcommand{\T}{{\mathbb T}}

\newtheorem{theorem}{Theorem}
\newtheorem{lemma}[theorem]{Lemma}

\newtheorem{remark}[theorem]{Remark}

\begin{document}

\title[Structure-preserving semi-convex-splitting scheme]{
Structure-preserving semi-convex-splitting \\ 
numerical scheme for a Cahn--Hilliard \\
cross-diffusion system in lymphangiogenesis}

\author[A. J\"ungel]{Ansgar J\"ungel}
\address{Institute of Analysis and Scientific Computing, Technische Universit\"at Wien, Wiedner Hauptstra\ss e 8--10, 1040 Wien, Austria}
\email{juengel@tuwien.ac.at}

\author[B. Wang]{Boyi Wang}
\address{Institute of Analysis and Scientific Computing, Technische Universit\"at Wien, Wiedner Hauptstra\ss e 8--10, 1040 Wien, Austria}
\email{boyi.wang@tuwien.ac.at}

\date{\today}

\thanks{The first author acknowledges partial support from
the Austrian Science Fund (FWF), grants P33010 and F65.
This work has received funding from the European
Research Council (ERC) under the European Union's Horizon 2020 research and
innovation programme, ERC Advanced Grant no.~101018153.}

\begin{abstract}
A fully discrete semi-convex-splitting finite-element scheme with stabilization for a Cahn--Hilliard cross-diffusion system is analyzed. The system consists of parabolic fourth-order equations for the volume fraction of the fiber phase and solute concentration, modeling pre-patterning of lymphatic vessel morphology. The existence of discrete solutions is proved, and it is shown that the numerical scheme is energy stable up to stabilization, conserves the solute mass, and preserves the lower and upper bounds of the fiber phase fraction. Numerical experiments in two space dimensions using FreeFEM illustrate the phase segregation and pattern formation. 
\end{abstract}



\maketitle


\section{Introduction}

In this paper, we suggest a thermodynamically consistent cross-diffusion system for lymphangiogenesis, based on the model of \cite{RoFo08}. Lymphangiogenesis is defined as the formation of new lymphatic vessels by lymphatic endothelial cells sprouting from existing vessels \cite{TaAl10}. Still, it may also occur in a different way, for instance, by migration of lymphatic endothelial cells in the direction of interstitial flow \cite{RBS06}. Our model describes the pre-patterning of lymphatic vessel morphology in collagen gels. The objective is the design of a structure-preserving fully discrete finite-element scheme and the existence of discrete solutions.

\subsection{Model equations}

The dynamics are assumed to be given by the collagen volume fraction and the solute concentration (protons, enzymes, nutrients) in a collagen implant, as experimentally realized in \cite{BoSw03}.
Assuming that the collagen implant consists of two phases, the fiber phase with (collagen) volume fraction $\phi$ and the fluid phase with volume fraction $1-\phi$, and that the solute is present in the fluid phase of the implant, the equations for the fiber phase volume fraction $\phi(x,t)$ and the solute or nutrient concentration $c(x,t)$ are given, according to \cite{RoFo08}, by
\begin{align}
  \pa_t\phi &= \diver\big(m(\phi)(\na\mu-c\na h_c(\phi,c))\big), 
  \label{1.phi} \\
  \pa_t c &= -\diver\big(cm(\phi)(\na\mu-c\na h_c(\phi,c))\big) 
  + \diver(g(c) \na h_c(\phi,c)), \label{1.c} \\
  \mu &= -\eps\Delta\phi + \eps^{-1}f(\phi) + h_\phi(\phi,c)\quad\mbox{in }{\T^d},\ t>0.
  \label{1.mu}
\end{align}
Here, $\T^d$ ($d\ge 1$) is the $d$-dimensional torus,
the degenerate mobility is given by
$$
  m(\phi)=\phi^2(1-\phi)^2,
$$
the diffusion coefficient $g(c)$ is nonnegative and satisfies $g(0)=0$ (which means that the diffusion in \eqref{1.c} is degenerate),
and
$$
  f(\phi)=\ln\frac{\phi}{1-\phi} + \frac{\theta_0}{2}(1-2\phi), \quad h(\phi,c)=\frac{c^2}{2} + c(1-\phi)
$$
are the derivative of the interaction energy and the nutrient energy \cite[(2.63)]{GKT22}, respectively. Furthermore, $h_c=\pa h/\pa c$, $h_\phi=\pa h/\pa\phi$ are partial derivatives, $\mu$ is called the chemical potential associated to the fiber phase, and $\theta_0>0$ and $\eps>0$ are some parameters. We refer to Section \ref{sec.model} for details on the model.
We impose the initial conditions
\begin{align*}
  \phi(\cdot,0) = \phi^0,\quad c(\cdot,0)=c^0 \quad\mbox{in }{\T^d}.
\end{align*}
We may also consider no-flux boundary conditions for a bounded domain in $\R^d$, but we assume periodic boundary conditions for the sake of simplicity.

Model \eqref{1.phi}--\eqref{1.mu} is a fourth-order cross-diffusion system with the following features. If the chemical potential $\mu$ is constant, the diffusion matrix associated to the variables $(\phi,c)$ has a vanishing eigenvalue. This issue indicates that it is more convenient to work with thermodynamic variables, which makes the diffusion matrix positive (semi-) definite. Furthermore, if the nutrient energy is constant, we obtain the Cahn--Hilliard equation for phase separation with a nonconvex energy. 

The key of our numerical analysis is the observation that \eqref{1.phi}-\eqref{1.mu} possesses the free energy
\begin{align}
  E(\phi,c) = \int_{\T^d}\bigg(\frac{\eps}{2}|\na\phi|^2 + \frac{1}{\eps}F(\phi) + h(\phi,c)\bigg)dx, \label{1.E}
\end{align}
consisting of the correlation, interaction, and nutrient energies. The interaction energy density
\begin{equation*}
  F(\phi) = \phi\ln\phi + (1-\phi)\ln(1-\phi) - \frac{\theta_0}{2}\phi(\phi-1)
\end{equation*}
is the difference of two convex functions, $\phi\mapsto\phi\ln\phi + (1-\phi)\ln(1-\phi)$ and $\phi\mapsto(\theta_0/2)\phi(\phi-1)$.
The chemical potential $\mu$ is the variational derivative of $E$ with respect to $\phi$ with $f=F'$.

A computation, detailed formally in Section \ref{sec.model} and made rigorous on the discrete level in Theorem \ref{thm.ex}, shows that
\begin{equation}\label{1.ei}
  \frac{dE}{dt} + \int_{\T^d}\big(m(\phi)|\na\mu-c\na h_c(\phi,c)|^2 
  + g(c)|\na h_c(\phi,c)|^2\big)dx = 0.
\end{equation}
The energy provides bounds for $\na\phi$ and $c$ in $L^2(\T^d)$, while the energy dissipation gives bounds for $\sqrt{m(\phi)}\na\phi$ and $\sqrt{g(c)}\na c$ only under some conditions on $g(c)$. An energy structure cannot be derived for the original model of \cite{RoFo08}, where the nutrient flux is given by $D(\phi)\na c$ and not by $g(c)\na h_c=g(c)(h_{c\phi}\na\phi+h_{cc}\na c)$ as in our model. The additional gradient $\na\phi$ is necessary to compensate the cross-diffusion terms in \eqref{1.phi}. This is not surprising from a thermodynamic viewpoint, and, as explained in Section \ref{sec.model}, the form $\na h_c$ follows from thermodynamic principles.

This paper aims to design a numerical scheme that preserves the physical properties of the model, namely mass conservation, energy stability, and the bounds $0\le\phi\le 1$. This is not trivial, since the energy is nonconvex and the higher-order and cross-diffusion structure does not allow for an application of a (discrete) maximum principle to conclude the bounds for $\phi$. We overcome these issues by using a semi-convex-splitting scheme to achieve energy stability and by exploiting the singularities of $f(\phi)$ at $\phi=0$ and $\phi=1$ to prove the lower and upper bounds.

\subsection{Stabilized semi-convex-splitting scheme}

The convex-splitting scheme was originally proposed in \cite{ElSt93} and revitalized in \cite{Eyr98}. Based on the classical convex-splitting scheme, the idea of this paper is first to write the interaction and nutrient energies as the difference of two functions $F(\phi)= F_1(\phi) - F_2(\phi)$ and $h(\phi,c) = h_1(c)-h_2(\phi,c)$ respectively, where $F_1, F_2, h_1$ are convex functions and $h_2(\phi,c)$ has a special structure, given by
\begin{equation*}
\begin{aligned}
  & F_1(\phi) = \phi\ln\phi + (1-\phi)\ln(1-\phi), \quad
	F_2(\phi) = \frac{\theta_0}{2}\phi(\phi-1), \\
	& h_1(c) = \frac12 c^2, \quad h_2(\phi, c) = c(\phi-1),
\end{aligned}
\end{equation*}
and second to treat $F'_1(\phi)=f_1(\phi)$, $h_{1,c}(c)$, $h_{2,\phi}(c)$ implicitly and $F'_2(\phi)=f_2(\phi)$, $h_{2,c}(\phi)$ explicitly. Typically, the time derivative is discretized by the backward Euler method, but also second-order convex-splitting schemes have been suggested in the literature; see, {e.g., \cite{DWW16,DWZZ20,LJWZ22}}. To recall the technique, we write \eqref{1.phi}--\eqref{1.mu} as the formal gradient flow
$$
  \pa_t u = \diver\big(M(u)\na\delta E(u)\big), \quad\mbox{where }
	u=\begin{pmatrix} \phi \\ c \end{pmatrix}.
$$
Here, the (symmetric, positive semidefinite) mobility matrix reads as
$$
	M(u) = \begin{pmatrix} m(\phi) & -cm(\phi) \\ -cm(\phi) & g(c)+c^2m(\phi)\end{pmatrix},
$$
and $\delta E(u)=(\mu,h_c)$ is the variational derivative of the energy with respect to $u$. We write $E(c,\phi)=E_1-E_2$, where
\begin{align*}
  E_1 = \int_{\T^d}\bigg(\frac{\eps}{2}|\na\phi|^2
  + \frac{1}{\eps}F_1(\phi) + h_1(c)\bigg)dx, \quad
  E_2 = \int_{\T^d}\bigg(\frac{1}{\eps}F_2(\phi) + h_2(\phi,c)\bigg)dx.
\end{align*}
The semidiscrete backward Euler semi-convex-splitting scheme with stabilization reads as
\begin{equation}\label{1.gf}
\begin{aligned}
  & u^{n+1}-u^n + \sigma\tau^2(\mu^{n+1}-\mu^n,0)^T 
  = \tau\diver\big(M(u^n)\na
  (\delta E_1(u^{n+1})-\delta E_2(u^{n,*})\big), \\
  & \mu^{n+1} = -\eps\Delta\phi^{n+1} 
  + \eps^{-1}\big(F_1'(\phi^{n+1})-F_2'(\phi^n)\big) 
  + h_{2\phi}(c^{n+1}),
\end{aligned}
\end{equation}
where $u^{n,*}=(\phi^n, c^{n+1})$, $\tau>0$ is the time step size, and $\sigma>0$ is a given constant. Because of the presence of $u^{n,*}$, the discretization is called a {\em semi}-convex-splitting scheme. The stabilization with parameter $\sigma>0$ is introduced to obtain an $L^2(\T^d)$ norm of $\mu^n$. The stabilization term is crucial for our finite-element analysis to deal with the degeneracy. In fact, energy inequality \eqref{1.ei} does not provide a bound for $\na\mu^n$ since $m(\phi)=0$ at $\phi=0$ and $\phi=1$. 

\subsection{Energy stability and physical bounds}

We claim that scheme \eqref{1.gf} is energy stable up to stabilization. For this, we first observe that the convexity of $E_1$ and the special structure of $E_2$ imply that
\begin{align*}
  \int_{\T^d}(u^{n+1}-u^n)\cdot\delta E_1(u^{n+1}) dx
  &\ge E_1(u^{n+1})-E_1(u^n), \\
  -\int_{\T^d}(u^{n+1}-u^n)\cdot\delta E_2(u^{n,*}) dx
  &\ge -\big(E_2(u^{n+1})-E_2(u^n)\big),
\end{align*}
where the last inequality follows from the identity
$$
  (\bar\phi-\phi)h_{2c}(\phi)+(\bar{c}-c)h_{2c}(\bar\phi)
  = h_2(\bar\phi,\bar{c}) - h_2(\phi,c).
$$
The stabilization term satisfies
\begin{align*}
  \sigma\tau^2\int_{\T^d}&(\mu^{n+1}-\mu^n,0)^T\cdot 
  (\delta E_1(u^{n+1})-\delta E_2(u^{n,*}))dx \\
  &= \sigma\tau^2\int_{\T^d}(\mu^{n+1}-\mu^n)\mu^{n+1}dx
  \ge \frac{\sigma\tau^2}2
  \big(\|\mu^{n+1}\|^2_{L^2(\T^d)}-\|\mu^{n}\|^2_{L^2(\T^d)}\big).
\end{align*}
The test function $\delta E^{n+1,n}:=\delta E_1(u^{n+1})-\delta E_2(u^{n,*})$ in the weak formulation of \eqref{1.gf} and the positive definiteness of $M(u^n)$ yield
\begin{align*}
  0 &\ge
  -\int_{\T^d}\na(\delta E^{n+1,n})^T M(u^n)\na(\delta E^{n+1,n})dx\\
  &= \int_{\T^d}\big((u^{n+1}-u^n)
  + \sigma\tau^2(\mu^{n+1}-\mu^n,0)^T\big)\cdot \delta E^{n+1,n} dx \\
  &\ge E(u^{n+1})-E(u^n)
  + \frac{\sigma\tau^2}{2}
  \big(\|\mu^{n+1}\|^2_{L^2(\T^d)}-\|\mu^{n}\|^2_{L^2(\T^d)}\big),
\end{align*}
and consequently 
$$
  E(u^{n+1}) + \frac{\sigma\tau^2}{2}\|\mu^{n+1}\|^2_{L^2(\T^d)}
  \le E(u^n) + \frac{\sigma\tau^2}{2}\|\mu^{n}\|^2_{L^2(\T^d)},
$$
This shows that the scheme is energy stable up to stabilization.


{Inspired by \cite{BBH2000} and with the aim to verify the physical bounds $0<\phi<1$, we consider a practical semi-convex-splitting scheme based on the idea explained above, where numerical quadrature is applied. With this idea, we prove a uniformly estimate for $|f_{1,\delta}(\phi_\delta)|^2$, where $f_{1,\delta}(\phi_\delta)$ is a regularization of $f_1(\phi)$ defined only in the interval $(0,1)$. This uniform estimate yields $0<\phi<1$ for the limit function $\phi$ of $(\phi_\delta)$.}

\subsection{Main results and state of the art}

Our main results are
\begin{itemize}
\item[(i)] the derivation of system \eqref{1.phi}--\eqref{1.mu} from thermodynamic principles, different from \cite{RoFo08},
\item[(ii)] the existence of a finite-element solution to the fully discrete semi-convex-splitting scheme,
\item[(iii)] the proof of energy stability up to stabilization, and
\item[(iv)] the proof of lower and upper bounds for the fiber phase fraction.
\end{itemize}
Moreover, we present some numerical tests in two space dimensions showing the phase separation expected for Cahn--Hilliard-type equations.

The model considered in this paper contains several mathematical difficulties: a degenerate mobility $m(\phi)$, a degenerate diffusion coefficient $g(c)$, cross-diffusion terms, and fourth-order derivatives. To deal with the finite-element approximation of the cross-diffusion equations, we exploit the thermodynamic structure of the model. The finite-element approach also has some limitations. First, we cannot prove the nonnegativity of the discrete concentrations, although the continuous equation preserves this property. This issue is well-known in finite-element theory, which we discuss in Remark \ref{rem.pos}. Second, we do not fully exploit the energy dissipation to obtain gradient bounds involving the chemical potential $\mu$, but instead we require a stabilization term to obtain low-order estimates for $\mu$. Gradient bounds for $\mu$ then follow from inverse inequalities. Therefore, our estimates are not uniform in the stabilization parameter and the mesh size, which prevents a convergence analysis. 

We finish the introduction by discussing the state of the art.
The modeling of lymphangiogenesis is quite recent. A system of ODEs was presented in \cite{BPS15} and extended in \cite{BPS16} to include spatial variations, describing the dynamics observed in wound healing lymphangiogenesis. The work \cite{FrLo05} analyzed a diffusion system with haptotaxis and chemotaxis terms for tumor lymphangiogenesis. Lymphangiogenesis processes in zebrafish embryos were modeled by reaction--diffusion--convection equations in \cite{WeRo17}. The collagen pre-patterning caused by interstitial fluid flow was described by Cahn--Hilliard-type equations in \cite{RoFo08}. The optimal structure of the lymphatic capillary network was studied in \cite{RoSw11} using homogenization theory. It was found that a hexagonal network is optimal in terms of fluid drainage, which is the structure found in mouse tails and human skin. A review on lymphangiogenesis models is given in \cite[Section 4]{MaBl12}.

Numerical schemes for Cahn--Hilliard equations are usually based on convex splitting
(see, e.g., \cite{GaWa12}). Another idea, still based on the energy gradient structure of the equations, is due to \cite{Fur01} using the so-called discrete variational derivative method. More recent papers analyze second-order convex-splitting schemes; see, e.g., \cite{FuYa22,LJWZ22}. In particular, energy-stable finite-difference \cite{CWWW19}, compact finite-difference \cite{LeSh19}, and mixed finite-element discretizations \cite{DWW16} have been investigated.

Cross-diffusion systems with Cahn--Hilliard terms have been suggested to model the dynamics in biological membranes \cite{GKRR16} and for tumor growth \cite{RSS23}. In these models, the cross-diffusion is of Keller--Segel type and thus, the diffusion matrix is triangular. Fourth-order degenerate cross-diffusion systems with diagonal mobility matrix were analyzed in \cite{MaZi17}. In \cite{HJT23}, Maxwell--Stefan models for fluid mixtures with full diffusion matrix and Cahn--Hilliard-type chemical potentials were investigated. Finally, a Cahn--Hilliard cross-diffusion model arising in physical vapor deposition was analyzed in \cite{EMP21} and numerically discretized in \cite{CEMP23}.

Up to our knowledge, the mathematical study of cross-diffusion Cahn--Hilliard equations like \eqref{1.phi}--\eqref{1.mu} is new. In particular, no energy-stable schemes seem to exist for such systems. The originality of the present paper is the thermodynamic modeling and the extension of (semi-) convex-splitting schemes to the cross-diffusion context.

The paper is organized as follows. System \eqref{1.phi}--\eqref{1.mu} is formally derived from thermodynamic principles in Section \ref{sec.model}. We present the fully discrete semi-convex-splitting scheme and the main existence result in Section \ref{sec.scheme}. Section \ref{sec.ex} is concerned with the proof of the main theorem. Finally, numerical experiments in two space dimensions are given in Section \ref{sec.num}.


\section{Thermodynamic derivation of the model}\label{sec.model}

We assume that the collagen implant consists of two phases, the fiber phase $\phi\in[0,1]$ and the fluid phase $1-\phi\in[0,1]$, driven by the conservation laws
\begin{align}\label{2.cons}
  \pa_t\phi + \diver(\phi v_{\rm fiber}) = 0, \quad
	\pa_t(1-\phi) + \diver((1-\phi)v_{\rm fluid}) = 0\quad\mbox{in }{\T^d},
\end{align}
where ${\T^d}\subset\R^d$ is the $d$-dimensional torus, and $v_{\rm fiber}$ and $v_{\rm fluid}$ are the fiber and fluid velocity, respectively. The following arguments hold true for bounded domains $\Omega\subset\R^d$ and no-flux boundary conditions.
Generally, the phase averaged velocity $v$ is given by $v=\phi v_{\rm fiber} + (1-\phi)v_{\rm fluid}$. Then the sum of both equations in \eqref{2.cons} implies the incompressibility condition $\diver v=0$. We suppose as in \cite{RoFo08} that the phase averaged velocity vanishes, $v=0$, so that the fiber and fluid velocities are related according to
$$
  \phi v_{\rm fiber} + (1-\phi)v_{\rm fluid} = 0.
$$
The solute (or nutrient) concentration is assumed to be driven by the conservation equation
\begin{equation}\label{2.nutr}
  \pa_t c + \diver(c(1-\phi)v_{\rm fluid}) = \diver J_c,
\end{equation}
where $J_c$ is the nutrient flux, which is determined later. In particular, we neglect source or sink terms, which can be justified by the observation that typically protons, nutrients, etc.\ are abundant in the solute.

Expressions for $v_{\rm fiber}$ and $J_c$ are derived from the second law of thermodynamics. For this, we suppose that the energy consists of the correlation, interaction, and nutrient energies,
$E=E_{\rm corr}+E_{\rm inter}+E_{\rm nutr}$, where
$$
  E_{\rm corr} = \frac{\eps}{2}\int_{\T^d} |\na\phi|^2dx, \quad
	E_{\rm inter} = \frac{1}{\eps}\int_{\T^d} F(\phi)dx, \quad
	E_{\rm nutr} = \int_{\T^d} h(\phi,c)dx,
$$
and $\eps>0$ is related to the correlation length. The energy density $F(\phi)$ is given by  the Flory--Huggins expression \cite{Flo42,Hug41}
\begin{equation*}
  F(\phi) = \phi\ln\phi + (1-\phi)\ln(1-\phi) - \frac{\theta_0}{2}\phi(\phi-1),
\end{equation*}
where $\theta_0>0$ is the Flory--Huggins mixing parameter. The first two terms represent the thermodynamic energy, and the last term favors the states $\phi=0$ and $\phi=1$. The nutrient energy density is taken as in \cite[(2.63)]{GKT22}:
\begin{equation*}
  h(\phi,c) = \frac12 c^2 + c(1-\phi).
\end{equation*}
The first term increases the energy in the presence of nutrients, and the second term can be interpreted as a chemotaxis energy, which accounts for interactions between the nutrient and the fiber phase.

According to the second law of thermodynamics, the energy dissipation $-dE/dt$ should be nonnegative \cite[Axiom (IV)]{BoDr15}. To compute this expression, we introduce the (fiber) chemical potential by $\mu=\delta E/\delta\phi= -\eps\Delta\phi + \eps^{-1} f(\phi) + h_\phi$, recalling that $f=F'$. Then a formal computation gives
\begin{align}\label{2.dEdt}
  -\frac{dE}{dt} &= -\int_{\T^d}\big((-\eps\Delta\phi + \eps^{-1}f(\phi) + h_\phi)\pa_t\phi
	+ h_c\pa_t c\big)dx = -\int_{\T^d}(\mu\pa_t\phi + h_c\pa_t c)dx \\
	&= -\int_{\T^d}\big(\phi v_{\rm fiber}\cdot\na\mu + c(1-\phi)v_{\rm fluid}\cdot\na h_c
	- J_c\cdot\na h_c\big)dx \nonumber \\
	&= \int_{\T^d}\big(-\phi v_{\rm fiber}\cdot(\na\mu - c\na h_c) + J_c\cdot\na h_c\big)dx,
	\nonumber
\end{align}
where we used $(1-\phi)v_{\rm fluid}=-\phi v_{\rm fiber}$ in the last step. The condition $-dE/dt\ge 0$ restricts the choice of constitutive relations for $v_{\rm fiber}$ and $J_c$.
Supposing that each product on the right-hand side of \eqref{2.dEdt} is nonnegative \cite[Axiom (IV) (ii)]{BoDr15}, an admissible choice is
\begin{equation}\label{2.vfiber}
  \phi v_{\rm fiber} = -D_{\rm fiber}(\phi,c)(\na\mu-c\na h_c), \quad
	J_c = D_{\rm fluid}(\phi,c)\na h_c
\end{equation}
for some coefficients $D_{\rm fiber}(\phi,c)$ and $D_{\rm fluid}(\phi,c)$.
A more general choice is given by the linear combination
$$
  \phi v_{\rm fiber} = -D_{11}(\na\mu-c\na h_c) + D_{12}\na h_c, \quad
	J_c = -D_{12}(\na\mu-c\na h_c) + D_{22}\na h_c,
$$
with a positive semidefinite diffusion matrix $(D_{ij})=(D_{ij}(\phi,c))\in\R^{2\times 2}$, but we prefer \eqref{2.vfiber} for the sake of simplicity. We specify our model by choosing 
$D_{\rm fiber}(\phi,c)=m(\phi)=\phi^2(1-\phi)^2$ and $D_{\rm fluid}(\phi,c)=g(c)$. Then the final system \eqref{2.cons}--\eqref{2.nutr} equals \eqref{1.phi}--\eqref{1.mu} and the energy equality is given by \eqref{1.ei}.

This corresponds to the model of \cite{RoFo08} with two differences. First, the model of \cite{RoFo08} includes the elastic energy density and not the mixing energy $(\theta_0/2)\phi(1-\phi)$. The elastic energy involves the number of monomers between cross-links, linking one polymer chain to another, and depends on the concentration $c$. Second, the diffusion flux in \cite{RoFo08} is given by $J_c=D(\phi)\na c$, leading to a diffusion equation for the solute concentration, which is driven by the fluid velocity $v_{\rm fluid}$. However, since $v_{\rm fluid}$ is coupled to the equation for the fiber phase, this choice is thermodynamically not correct. We need a choice like $J_c=D_{\rm fluid}\na h_c$ to ensure the thermodynamic consistency. From a modeling viewpoint, this is the main difference between our model and the model of \cite{RoFo08}.


\section{Stabilized semi-convex-splitting scheme}\label{sec.scheme}


Let $N\in\N$, let $T>0$ be the final time, $\tau=T/N$ the time step size, and $\sigma>0$ the stabilization parameter. {We use a quasi-conform triangulation $\{\mathcal{T}^h\}$ with disjoint open simplices $\kappa$ satisfying $\mathbb{T}^d = \cup_{\kappa\in \mathcal{T}^h}\overline\kappa$ and introduce the space $X_h$ of continuous linear finite elements on the torus, where $h>0$ is a measure of the size of the spatial grid (not to be confused with the energy function $h(\phi,c)$). Denote the set of the nodes by $J$ and the coordinates of these nodes by $\{x_j\}_{j\in J}$. Let $\{p_j\}$ be the set of standard basis functions for $X_h$ defined by $p_i(x_j)=\delta_{ij}$. We introduce the standard nodal interpolation operator $I^h:C(\T^d)\to X_h$ by $(I^h q)(x_j)=q(x_j)$ for all $q\in C(\T^d)$ and $j\in J$. We also introduce the discrete semi-inner product and the induced semi-norm on $C(\T^d)$ by
\begin{align*}
  (u,v)_h = \int_{\T^d}I^h[u(x)v(x)]dx 
  := \sum_{j\in J}\beta_ju(x_j)v(x_j),\quad 
  |u|_h=\sqrt{(u,u)_h} \quad\mbox{for } u,v\in C(\T^d),
\end{align*}
where $\beta_j=\int_{\T^d}p_j dx$. The initial data is interpolated on $X_h$, and we write $\phi_h^0=I^{h}\phi^0$, $c_h^0=I^{h}c^0$.}

As explained in the introduction, the idea of the convex-splitting scheme is to write the energy as the difference of two convex functions and treat one function implicitly and the other one explicitly. We define
\begin{align*}
  & f(\phi) = f_1(\phi) - f_2(\phi), \quad\mbox{where }
	f_1(\phi) = \ln\phi - \ln(1-\phi), \quad f_2(\phi) = \frac{\theta_0}{2}(2\phi-1), \\
	& h(\phi,c) = h_1(c) - h_2(\phi,c), \quad\mbox{where }
	h_1(c) = \frac12 c^2, \quad h_2(\phi,c) = c(\phi-1).
\end{align*}
Then $h_{1,c}(c)=c$, $h_{2,c}(\phi,c)=\phi-1$,  $h_\phi(\phi,c)=-h_{2\phi}(\phi,c)=-c$, and consequently
$h_{1,c}(c_h^{n+1})-h_{2,c}(\phi_h^n) = c_h^{n+1}+1-\phi_h^n$.
We recall that the diffusivity $g(c)$ is assumed to be nonnegative and $g(0)=0$. An example is $g(c)=c^2$.
We wish to solve
\begin{align}
   & 0 = \frac{1}{\tau}\int_{\T^d}I^h[(\phi_h^{n+1}-\phi_h^n)\zeta_h] dx
	+ \sigma\tau\int_{\T^d}I^h[(\mu_h^{n+1}-\mu_h^n)\zeta_h] dx 
	\label{2.phi} \\
  &\phantom{xxx} + \int_{\T^d}m(\phi_h^n)\big(\na\mu_h^{n+1} 
    - c_h^n\na(c_h^{n+1}+1-\phi_h^n)\big)
	\cdot\na\zeta_h dx, \nonumber\\
  &0 = \frac{1}{\tau}\int_{\T^d}I^h[(c_h^{n+1}-c_h^n)\xi_h] dx
    + \int_{\T^d} g(c_h^n)\na(c_h^{n+1}+1-\phi_h^n)\cdot\na\xi_h dx \label{2.c} \\
  &\phantom{xxx} - \int_{\T^d}c_h^n m(\phi_h^n)\big(\na\mu_h^{n+1}
	-c_h^n\na(c_h^{n+1}+1-\phi_h^n)\big)\cdot\na\xi_h dx, \nonumber \\
  &\int_{\T^d}I^h[\mu_h^{n+1}\chi_h] dx 
  = \eps\int_{\T^d}\na\phi_h^{n+1}\cdot\na\chi_h dx \label{2.mu} \\
  &\phantom{xxx}+ \int_{\T^d}I^h\bigg[\bigg(\frac{1}{\eps}
  (f_1(\phi_h^{n+1})-f_2(\phi_h^n))-c_h^{n+1}\bigg)\chi_h\bigg] dx
  \nonumber
\end{align}
for all $(\zeta_h,\xi_h,\chi_h)\in X_h^3$ and $0\le n\le N-1$. Notice that equations \eqref{2.phi} and \eqref{2.c} are linear in $(\phi_h^{n+1},c_h^{n+1},\mu_h^{n+1})$, which simplifies the numerical implementation. We only need to implement the Newton method for the semilinear equation \eqref{2.mu}.

Our main result reads as follows.

\begin{theorem}[Existence of a discrete solution]\label{thm.ex}
Let $\sigma>0$ be a given constant, $(\phi_h^0,c_h^0)\in X_h^2$ with $0<\phi_h^0<1$ in ${\T^d}$, and let the time step size $\tau>0$ be sufficiently small. Let $g:\R\to\R$ be nonnegative and $g(0)=0$. Then, for all $n=1,\ldots,N$,
there exists a {unique solution $(\phi_h^{n+1},c_h^{n+1},\mu_h^{n+1})\in X_h^3$} to \eqref{2.phi}--\eqref{2.mu}
satisfying $0<\phi_h^n<1$ a.e. in $\T^d$ and
\begin{align}\label{2.Eineq}
  E^{n+1} - E^n &\le -\tau\int_{\T^d}m(\phi_h^{n})\big|\na\mu_h^{n+1}
	- c_h^n\na(c_h^{n+1}+1-\phi_h^n)\big|^2dx \\
	&\phantom{xx}- \tau\int_{\T^d}g(c_h^n)
	\big|\na(c_h^{n+1}+1-\phi_h^n)\big|^2 dx,\nonumber
\end{align}
where the discrete energy is given by
{$$  
  E^n= \int_{\T^d}\bigg\{\frac{\eps}{2}|\na\phi^n|^2 
  + I^h\bigg[\frac{1}{\eps}F(\phi^n)+ h(\phi^n,c^n)\bigg]\bigg\}dx 
  + \frac{\sigma\tau^2}{2}\int_{\T^d}I^h[|\mu_h^{n}|^2]dx.
$$}
\end{theorem}

The smallness of the time step size $\tau$ is also required in \cite[Theorem 4.1]{GKT22} to show the existence of discrete solutions. This restriction is not needed to derive the energy stability \eqref{2.Eineq}. Inequality \eqref{2.Eineq} is formally derived by using $\tau\mu_h^{n+1}$, $h_c(\phi^n,c^{n+1}) = c_h^{n+1}+1-\phi_h^n$, and $\phi_h^{n+1}-\phi_h^n$ as test functions in \eqref{2.phi}, \eqref{2.c}, and \eqref{2.mu} respectively. Here, we exploit the fact that $h(\phi^n,c^{n+1})$ is an element of the finite-element space, which requires an (at most) quadratic nutrient energy. For more general expressions, we need to interpolate the function, which results in correction terms whose estimation may lead to an approximate energy inequality only.

The proof of existence of solutions is based on the Brouwer fixed-point theorem, a priori estimates coming from the discrete energy inequality \eqref{2.Eineq}, and additional mesh-size depending estimates. To deal with the singularity of $f_1$ at $\phi\in\{0,1\}$, we regularize this function by some $f_{1,\delta}$. The energy inequality provides basic a priori estimates for the approximate sequences $(\phi_\delta^{n+1})$, $(\nabla\phi_\delta^{n+1})$, $(c_\delta^{n+1})$, and, thanks to the stabilization term, for $(\mu_\delta^{n+1})$. The energy dissipation terms cannot be easily exploited; therefore, we use inverse inequalities to infer bounds for $\na\mu_\delta^{n+1}$ and $\na c_\delta^{n+1}$, which are independent of $\delta$ but may depend on the mesh size.


It remains to establish the convergence of the nonlinear singular term $f_{1,\delta}(\phi_\delta^{n+1})$. To do this, we prove an uniform bound for $I^h[|f_{1,\delta}(\phi_\delta^{n+1})|^2]$. Thanking to the stabilization term with respect to $\mu$ which provides a uniform estimate for $\mu_\delta^{n+1}$, we deduce the integral bound for $I^h[|f_{1,\delta}(\phi_\delta^{n+1})|^2]$ from \eqref{2.mu}. Finally, arguing like in \cite{BBH2000}, the limit $\phi_h^{n+1}=\lim_{\delta\to 0}\phi_\delta^{n+1}$ satisfies the constraints $0<\phi_h^{n+1}<1$.

\begin{remark}\label{rem.pos}\rm
We discuss whether the nonnegativity of the discrete concentration $c_h$ can be expected. The discrete minimum or maximum principle cannot be proven from a Stampacchia truncation argument, since the test function $\min\{0,c_h^n\}$ for some $c_h^n\in X_h$ is generally not an element of the finite-element space. In fact, the discrete maximum principle generally does not hold for finite elements \cite{Vej04}. It is possible to prove a discrete maximum principle by requiring some conditions on the mesh, like acute triangulations \cite{KaKo05,XuZi99}. One idea is to define the test function $v_h(A)=\min\{0,c_h^n(A)\}$ at the nodal points $A$ \cite{WaZh12}. However, this function is not compatible with the time discretization and cannot be used for our equations. For general meshes, it was proved in \cite[Theorem~9]{JuUn05} that if the solution $c$ to the continuous problem is nonnegative then the finite-element solution $c_h^n$ satisfies $c_h^n\ge \min c_h^0 - Ch^\alpha$, where $C>0$ depends on the data and $\alpha>1$. Thus, the solution $c_h^n$ may be negative but it becomes positive if $\min c_0^h$ is positive and the mesh size $h$ is sufficiently small. Our numerical simulations in Section \ref{sec.num} confirm this statement.
\qed\end{remark}

\begin{remark}\label{rem.lit}\rm
{In the literature, the existence and positivity of discrete solutions to Cahn--Hilliard-type equations can be also proved by solving a minimization problem of a convex functional \cite{CWWW19,DWZZ19,DWZZ20} or by using a mass lumped finite-element method \cite{YCWWZ21}. In this paper, we are not considering a pure convex-splitting scheme like in \cite{CWWW19,DWZZ19,DWZZ20} such that the results in these papers cannot be directly used. Therefore, we provide an existence proof.}
\qed\end{remark}


\section{Proof of Theorem \ref{thm.ex}}\label{sec.ex}

We split the proof in several steps. {First, we prove the uniqueness of solutions.} Then we show the existence of solutions to a regularized scheme, avoiding possible singularities when dealing with the logarithm, derive an approximate energy inequality and further estimates uniform in the regularization parameters, and finally pass to the de-regularization limit using compactness arguments.

\subsection{Uniqueness of solutions}

{Let $n\geq0$ be fixed. We show the uniqueness of solutions to \eqref{2.phi}--\eqref{2.mu} satisfying $0\leq\phi_h^n<1$. Suppose that there are two different solutions $(\phi_{i},c_{i},\mu_{i})$ to \eqref{2.phi}--\eqref{2.mu} with $i=1,2$ satisfying  $\phi_{i}\in[a^n,b^n]\subset(0,1)$. Furthermore,, let $(\widehat\phi,\widehat c,\widehat\mu) 
= (\phi_1-\phi_2,c_1-c_2,\mu_1-\mu_2)$.
We substract \eqref{2.phi}--\eqref{2.mu} with $(\phi_{n+1},c_{n+1},\mu_{n+1})$ replaced by $(\phi_{i},c_{i},\mu_{i})$ for $i=1,2$, respectively, to find that
\begin{align}
	& 0 = \frac{1}{\tau}\int_{\T^d}I_h[\widehat\phi\zeta_h] dx
		+ \sigma\tau\int_{\T^d}I_h[\widehat\mu\zeta_h] dx 
		+ \int_{\T^d}m(\phi_h^n)\big(\na\widehat\mu 
		- c_h^n\na\widehat c\big)\cdot\na\zeta_h dx, \label{2.dphi}\\
	& 0 = \frac{1}{\tau}\int_{\T^d}I_h[\widehat c\xi_h] dx
		+ \int_{\T^d} g(c_h^n)\na\widehat c\cdot\na\xi_h dx
 		- \int_{\T^d}c_h^n m(\phi_h^n)\big(\na\widehat\mu
		-c_h^n\na\widehat c\big)\cdot\na\xi_h dx, \label{2.dc}  \\
	&\int_{\T^d}I_h[\widehat\mu\chi_h] dx 
		= \eps\int_{\T^d}\na\widehat\phi\cdot\na\chi_h dx
		+ \int_{\T^d}I_h\bigg[\bigg(\frac{1}{\eps}
		(f_1(\phi_1)-f_1(\phi_2))-\widehat c\bigg)\chi_h\bigg] dx.
		\label{2.dmu}
\end{align}
We choose the test functions $(\zeta_h,\xi_h,\chi_h) = (\widehat\mu,\widehat c,\widehat\phi)\in X_h^3$ and sum \eqref{2.dphi}, \eqref{2.dc}, and \eqref{2.dmu}$\times \tau^{-1}$:
\begin{align*}
  0 &= \int_{\T^d}\big(\eps\tau^{-1}|\na\widehat\phi|^2
  + I_h[\widehat c^2+\sigma\tau\widehat\mu^2]\big) dx 
  + \int_{\T^d}m(\phi_h^n)|\na\widehat\mu - c_h^n\na\widehat c|^2 dx, \\
  &\phantom{xx} + \int_{\T^d} g(c_h^n)|\na\widehat c|^2 dx 
  + \frac{1}{\tau}\int_{\T^d}I_h\bigg[\bigg(\frac{1}{\eps}
  (f_1(\phi_1)-f_1(\phi_2))-\widehat c\bigg)\widehat\phi\bigg] dx.
\end{align*}
Recalling that $\phi_{1,2}\in[a^n,b^n]\subset(0,1)$, the mean-value value theorem gives $f_1(\phi_1)-f_1(\phi_2) = f_1'(\xi)\widehat\phi$ for some $\xi\in(0,1)$. We deduce from $f_1'(\xi)\geq4$ for all $\xi\in(0,1)$ that $(f_1(\phi_1)-f_1(\phi_2))\hat\phi \geq 4\widehat\phi^2$. Therefore, using Young's inequality,
\begin{align*}
		0 &\geq\int_{\T^d}\bigg(\eps\tau^{-1}|\na\widehat\phi|^2
		+I_h\bigg[\frac12\widehat c^2
		+\sigma\tau\widehat\mu^2\bigg]\bigg) dx 
		+ \int_{\T^d}m(\phi_h^n)|\na\widehat\mu 
		- c_h^n\na\widehat c|^2 dx \\
		&\phantom{xx} + \int_{\T^d} g(c_h^n)|\na\widehat c|^2 dx 
		+ \frac{1}{\tau}\int_{\T^d}I_h\bigg[\bigg(\frac{4}{\eps}
		-\frac12\bigg)\widehat\phi^2\bigg] dx.
\end{align*}
Choosing $0<\eps<2,$ we conclude that $(\widehat\phi,\widehat c,\widehat\mu) =0$. This implies that  $(\phi_{1},c_{1},\mu_{1})(x_j)=(\phi_{2},c_{2},\mu_{2})(x_j)$ for all $j\in J$. Therefore $(\phi_{1},c_{1},\mu_{1})=(\phi_{2},c_{2},\mu_{2})$, finishing the proof.}

\subsection{Regularized problem and regularized energy inequality}

We first show the existence of a solution to a regularized scheme. Let $\delta\in(0,1/2)$ and
\begin{align}\label{func.f}
  F_{1,\delta}(\phi) &= \begin{cases}
	(1-\phi)\ln(1-\phi) + \frac{(\phi-\delta)^2}{2\delta} + (\ln\delta+1)(\phi-\delta)
	+ \delta\ln\delta &\mbox{if }\phi\le\delta, \\
	\phi\ln\phi + (1-\phi)\ln(1-\phi) &\mbox{if }\delta<\phi<1-\delta, \\
	\phi\ln\phi + \frac{(\phi-1+\delta)^2}{2\delta} - (\ln\delta+1)(\phi-1+\delta)
	+ \delta\ln\delta &\mbox{if }1-\delta\le\phi,
	\end{cases} \\
  f_{1,\delta}(\phi) &= \begin{cases}
	\delta^{-1}(\phi-\delta) + \ln\delta - \ln(1-\phi) & \mbox{if }\phi\le\delta, \\
	\ln\phi-\ln(1-\phi) & \mbox{if }\delta<\phi<1-\phi, \\
	\delta^{-1}(\phi-1+\delta) + \ln\phi - \ln\delta &\mbox{if }1-\delta\le\phi.
	\end{cases}
\end{align}
Then $F'_{1,\delta}=f_{1,\delta}$. Recall that $f_2(\phi)=\theta_0(2\phi-1)/2$. Given $(\phi_h^n,c_h^n,\mu_h^n)\in X_h^3$, we wish to find a solution
$(\phi_\delta^{n+1},c_\delta^{n+1},\mu_\delta^{n+1})\in X_h^3$ to the regularized problem
\begin{align}
  0 &= \frac{1}{\tau}\int_{\T^d}I^h[(\phi_\delta^{n+1}
  -\phi_h^n)\zeta_h] dx
  + \sigma\tau \int_{\T^d}I^h[(\mu_\delta^{n+1}-\mu_h^n)\zeta_h] dx 
  \label{3.phi} \\
  &\phantom{xx}+ \int_{\T^d}m(\phi_h^n)
	\big(\na\mu_\delta^{n+1} 
	- c_h^n\na(c_\delta^{n+1}+1-\phi_h^n)\big)
	\cdot\na\zeta_h dx, \nonumber \\
  0 &= \frac{1}{\tau}\int_{\T^d}I^h[(c_\delta^{n+1}-c_h^n)\xi_h] dx
    + \int_{\T^d}g(c_h^{n})\na(c_\delta^{n+1}+1-\phi_h^n)
    \cdot\na\xi_h dx \label{3.c} \\
  &\phantom{xx}- \int_{\T^d}c_h^n m(\phi_h^n)
    \big(\na\mu_\delta^{n+1}-c_h^n
	\na(c_\delta^{n+1}+1-\phi_h^n)\big)\cdot\na\xi_h dx, 
	\nonumber \\
  \int_{\T^d}I^h[\mu_\delta^{n+1}\chi_h] dx 
    &= \eps\int_{\T^d}\na\phi_\delta^{n+1}\cdot\na\chi_h dx
	+ \int_{\T^d}I^h\bigg[\bigg(\frac{1}{\eps}(f_{1,\delta}(\phi_\delta^{n+1})
	- f_2(\phi_h^n)) - c_\delta^{n+1}\bigg)\chi_h\bigg] dx \label{3.mu}
\end{align}
for all $(\zeta_h,\xi_h,\chi_h)\in X_h^3$. We show an energy inequality associated to the regularized scheme.

\begin{lemma}\label{lem.Edelta}
It holds for given $(\phi_h^n,c_h^n,\mu_h^n)$,
$(\phi_\delta^{n+1},c_\delta^{n+1},\mu_\delta^{n+1})\in X_h^3$ solving \eqref{3.phi}--\eqref{3.mu} that
\begin{align*}
  E^{n+1}_\delta - E^n_h &\le -\tau\int_{\T^d}m(\phi_h^n)\big|\na\mu_\delta^{n+1}
	- c_h^n\na(c_\delta^{n+1}+1-\phi_h^n)\big|^2 dx \\
	&\phantom{xx}- \tau\int_{\T^d}g(c_h^n)
	\big|\na(c_\delta^{n+1}+1-\phi_h^n)\big|^2dx,
\end{align*}
where
$$
  E^n_{h/\delta} = \int_{\T^d}\bigg(\frac{\eps}{2}
  |\na\phi_{h/\delta}^n|^2 + I^h\bigg[\frac{\sigma\tau^2}2|\mu_{h/\delta}^n|^2
  + \frac{1}{\eps}(F_{1,\delta}(\phi_{h/\delta}^n)
  -F_{2}(\phi_{h/\delta}^n)) 
  + h(\phi_{h/\delta}^n,c_{h/\delta}^n)\bigg]\bigg)dx.
$$
Moreover, the solution conserves the mass in the sense
$$
	\int_{\T^d} c_\delta^{n+1}dx = \int_{\T^d} c_h^n dx.
$$
\end{lemma}

\begin{proof}
The mass conservation property follows immediately by choosing
$\xi_h=1$ in \eqref{3.c}. Next, we choose the test function $\zeta_h=\mu_\delta^{n+1}\in X_h$ in \eqref{3.phi} and $\xi_h=c_\delta^{n+1}+1-\phi_h^n\in X_h$ in \eqref{3.c}. Here, we take advantage of the fact that $h_c(\phi_h^{n},c_\delta^{n+1})=c_\delta^{n+1}+1-\phi_h^n$ is linear in both arguments, which ensures that this function lies in $X_h$. An addition of \eqref{3.phi} and \eqref{3.c} with these test functions yields
\begin{align}
  \frac{1}{\tau}&\int_{\T^d}I^h[(\phi_\delta^{n+1}-\phi_h^n)
  \mu_\delta^{n+1} ]dx
	+ \frac{1}{\tau}\int_{\T^d}I^h[
	(c_\delta^{n+1}-c_h^n)(c_\delta^{n+1}+1-\phi_h^n)]dx \nonumber \\
  &\quad + \sigma\tau \int_{\T^d}I^h[(\mu_\delta^{n+1}-\mu_h^n)
  \mu_\delta^{n+1}] dx \label{3.aux} \\
	&= -\int_{\T^d}m(\phi_h^n)\big(\na\mu_\delta^{n+1}
	-c_h^n\na(c_\delta^{n+1}+1-\phi_h^n)\big)
	\cdot\na\mu_\delta^{n+1} dx \nonumber \\
   &\phantom{xx}+ \int_{\T^d} c_h^nm(\phi_h^n)\big(\na\mu_\delta^{n+1}-c_h^n
	\na(c_\delta^{n+1}+1-\phi_h^n)\big)
	\cdot\na(c_\delta^{n+1}+1-\phi_h^n)dx \nonumber \\
	&\phantom{xx}- \int_{\T^d}g(c_h^n)
	|\na(c_\delta^{n+1}+1-\phi_h^n)|^2 dx \nonumber \\
	&= -\int_{\T^d}m(\phi_h^n)\big|\na\mu_\delta^{n+1}
	-c_h^n\na(c_\delta^{n+1}+1-\phi_h^n)\big|^2 dx
	- \int_{\T^d}g(c_h^n)|\na(c_\delta^{n+1}+1-\phi_h^n)|^2 dx. 
	\nonumber
\end{align}
Finally, we choose the test function $\chi_h=(\phi_\delta^{n+1}-\phi_h^n)/\tau$ in \eqref{3.mu}. Then, with the elementary inequality $\eta'(z)(z-y)\ge \eta(z)-\eta(y)$ for convex functions $\eta$ and any $(y,z)$, the left-hand side of \eqref{3.aux} becomes
\begin{align}\label{3.aux2}
  \frac{1}{\tau}&\int_{\T^d}I^h[\mu_\delta^{n+1}
  (\phi_\delta^{n+1}-\phi_h^n)] dx
  + \frac{1}{\tau}\int_{\T^d}I^h[(c_\delta^{n+1}-c_h^n)
  (c_\delta^{n+1}+1-\phi_h^n)]dx \\
  &\phantom{xx}+ \sigma\tau \int_{\T^d}I^h[(\mu_\delta^{n+1}-\mu_h^n)\mu_\delta^{n+1}] dx
  \nonumber \\
	&= \frac{\eps}{\tau}\int_{\T^d}\na\phi_\delta^{n+1}\cdot
	\na(\phi_\delta^{n+1}-\phi_h^n)dx \nonumber \\
	&\phantom{xx}+ \frac{1}{\tau}\int_{\T^d}I^h\bigg[\bigg(\frac{1}{\eps}
	(f_{1,\delta}(\phi_\delta^{n+1})-f_2(\phi_h^n))
	- c_\delta^{n+1}\bigg)(\phi_\delta^{n+1}-\phi_h^n)\bigg]dx \nonumber \\
	&\phantom{xx}+ \frac{1}{\tau}\int_{\T^d}
	I^h[(c_\delta^{n+1}-c_h^n)(c_\delta^{n+1}+1-\phi_h^n)]dx  
	+ \sigma\tau \int_{\T^d}I^h[(\mu_\delta^{n+1}-\mu_h^n) 
	\mu_\delta^{n+1}] dx \nonumber\\
	&\ge \frac{\eps}{2\tau}\int_{\T^d}\big(|\na\phi_\delta^{n+1}|^2
	-|\na\phi_h^n|^2\big)dx
    +\frac{\sigma\tau}2\int_{\T^d}I^h\bigg[\big(|\mu_\delta^{n+1}|^2
    -|\mu_h^n|^2\big)\bigg]dx \nonumber \\
	&\phantom{xx}+ \frac{1}{\tau\eps}\int_{\T^d}I^h[\big((F_{1,\delta}
	(\phi_\delta^{n+1})-F_{1,\delta}(\phi_h^n)) 
	- (F_2(\phi_\delta^{n+1})-F_2(\phi_h^n))\big)]dx \nonumber \\
	&\phantom{xx}- \frac{1}{\tau}\int_{\T^d} I^h[c_\delta^{n+1}(\phi_\delta^{n+1}-\phi_h^n)]dx
	+ \frac{1}{\tau}\int_{\T^d}
	I^h[(c_\delta^{n+1}-c_h^n)(c_\delta^{n+1}+1-\phi_h^n)]dx. \nonumber
\end{align}
Because of $(c_\delta^{n+1}-c_h^n)c_\delta^{n+1}\ge \frac12((c_\delta^{n+1})^2-(c_h^n)^2)$, the last two integrals can be estimated according to
\begin{align*}
  - \frac{1}{\tau}&\int_{\T^d} I^h[c_\delta^{n+1}
  (\phi_\delta^{n+1}-\phi_h^n)]dx
  + \frac{1}{\tau}\int_{\T^d}I^h[(c_\delta^{n+1}-c_h^n)
  (c_\delta^{n+1}+1-\phi_h^n)]dx \\
  &= \frac{1}{\tau}\int_{\T^d}I^h[\big((c_\delta^{n+1}-c_h^n)
  c_\delta^{n+1} + c_\delta^{n+1}(1-\phi_\delta^{n+1}) 
  - c_h^n(1-\phi_h^n))\bigg]dx \\
  &\ge \frac{1}{\tau}\int_{\T^d}I^h[\big(h(\phi_\delta^{n+1},c_\delta^{n+1})
  - h(\phi_h^n,c_h^n)\big)]dx.
\end{align*}
Therefore, it follows from \eqref{3.aux2} that
\begin{align*}
  \frac{1}{\tau}\int_{\T^d}&I^h[(\phi_\delta^{n+1}-\phi_h^n)
  \mu_\delta^{n+1}] dx 
  + \frac{1}{\tau}\int_{\T^d}I^h[(c_\delta^{n+1}-c_h^n)
  (c_\delta^{n+1}+1-\phi_h^n)]dx \\
  &+ \sigma\tau\int_{\T^d}I^h[(\mu_\delta^{n+1}-\mu_h^n)\mu_\delta^{n+1}]dx 
  \ge \frac1\tau(E_\delta^{n+1}-E_h^n).
\end{align*}
Replacing the left-hand side by \eqref{3.aux} finishes the proof.
\end{proof}

\subsection{Solution to the regularized problem}

The existence of solutions to the regularized system is proved by means of the Brouwer fixed-point theorem applied to a mapping inspired by \cite[Section 4.4]{GKT22}.

\begin{lemma}
Let $0<\delta<1/2$ be suitably small and let $(\phi_h^n,c_h^n,\mu_h^n)\in X_h^3$ be given. Then there exists a solution $(\phi_\delta^{n+1},c_\delta^{n+1},\mu_\delta^{n+1})\in X_h^3$ to \eqref{3.phi}--\eqref{3.mu}.
\end{lemma}

\begin{proof}
We define the inner product
$$
  \langle(\phi_h,c_h,\mu_h),(\zeta_h,\xi_h,\chi_h)\rangle
	:= \int_{\T^d}I^h[(\phi_h\zeta_h + c_h\xi_h + \mu_h\chi_h)]dx
$$
on the Hilbert space $X_h^3$. Let $(\phi_h^n,c_h^n,\mu_h^n)\in X_h^3$ be given and introduce the mapping $S:X_h^3\to X_h^3$ by
\begin{align*}
  \langle & S(\phi_h,c_h,\mu_h),(\zeta_h,\xi_h,\chi_h)\rangle \\
	&= \int_{\T^d}\bigg(I^h\bigg[\frac{1}{\tau}(\phi_h-\phi_h^n)\zeta_h\bigg]
	+ m(\phi_h^n)(\na\mu_h-c_h^n\na(c_h+1-\phi_h^n))\cdot\na\zeta_h \\
    &\phantom{xx}+\sigma\tau I^h[(\mu_h-\mu_h^n)\zeta_h]\bigg)dx \\
	&\phantom{xx}+ \int_{\T^d}\bigg\{\eps\na\phi_h\cdot\na\chi_h + I^h\bigg[\bigg(-\mu_h + \frac{1}{\eps}
	(f_{1,\delta}(\phi_h)-f_2(\phi_h^n))-c_h\bigg)\chi_h\bigg]\bigg\}dx \\
	&\phantom{xx}+ \int_{\T^d}\bigg(I^h\bigg[\frac{1}{\tau}(c_h-c_h^n)(\xi_h+1-\phi_h^n)\bigg] \\
	&\phantom{xx}- c_h^nm(\phi_h^n)\big(\na\mu_h-c_h^n
    \na(c_h+1-\phi_h^n)\big)\cdot\na(\xi_h+1-\phi_h^n) \\
	&\phantom{xx}+ g(c_h^n)\na(c_h+1-\phi_h^n)\cdot\na(\xi_h+1-\phi_h^n)\bigg)dx
\end{align*}
for all $(\zeta_h,\xi_h,\chi_h)\in X_h^3$. A solution to \eqref{3.phi}--\eqref{3.mu} corresponds to a zero of $S$.

Let $L:X_h^3\to X_h^3$ be the linear transformation $L(\phi_h,c_h,\mu_h)=(\mu_h,c_h,\phi_h/\tau)$ with its inverse $L^{-1}(\zeta_h,\xi_h,\chi_h)=(\tau\chi_h,\xi_h,\zeta_h)$. Furthermore, let $R>0$. We suppose by contradiction that the continuous mapping $S\circ L^{-1}$ has no zeros in the ball $B_R:=\{(\zeta_h,\xi_h,\chi_h)\in X_h^3:\|(\zeta_h,\xi_h,\chi_h)\|_2\le R\}$, where $\|(\zeta_h,\xi_h,\chi_h)\|_2^2=\langle(\zeta_h,\xi_h,\chi_h),(\zeta_h,\xi_h,\chi_h)\rangle$. Furthermore, similarly as in \cite[Section 4.4]{GKT22}, we define the continuous mapping $G_R:B_R\to\pa B_R$ by
$$
  G_R(\zeta_h,\xi_h,\chi_h) 
  := -R\frac{(S\circ L^{-1})(\zeta_h,\xi_h,\chi_h)}{
	\|(S\circ L^{-1})(\zeta_h,\xi_h,\chi_h)\|_2} \quad\mbox{for }(\zeta_h,\xi_h,\chi_h)\in B_R.
$$
By Brouwer's fixed-point theorem, there exists a fixed point $(\zeta_h,\xi_h,\chi_h)\in B_R$ of $G_R$ such that $\|G_R(\zeta_h,\xi_h,\chi_h)\|_2=\|(\zeta_h,\xi_h,\chi_h)\|_2=R$. By definition of $L$, there exists $(\phi_h,c_h,\mu_h)\in X_h^3$ such that $L(\phi_h,c_h,\mu_h)=(\zeta_h,\xi_h,\chi_h)$. Then, since $\zeta_h=\mu_h$, $\xi_h=c_h$, and $\chi_h=\phi_h/\tau$,
\begin{align*}
  &\langle S(\phi_h,c_h,\mu_h),(\zeta_h,\xi_h,\chi_h)\rangle \\
	&= \int_{\T^d}\bigg(\frac{1}{\tau}I^h[(\phi_h-\phi_h^n)\mu_h]
	+ m(\phi_h^n)\big(\na\mu_h-c_h^n\na(c_h+1-\phi_h^n)\big)
	\cdot\na\mu_h \\
    &\phantom{xx}+\sigma\tau I^h[(\mu_h-\mu_h^n)\mu_h]\bigg)dx \\
	&\phantom{xx}+ \int_{\T^d}\bigg\{\eps\na\phi_h\cdot\na\frac{\phi_h}{\tau}
	+ I^h\bigg[\bigg(-\mu_h+\frac{1}{\eps}(f_{1,\delta}(\phi_h)
	-f_2(\phi_h^n))-c_h\bigg)\frac{\phi_h}{\tau}\bigg]\bigg\}dx \\
	&\phantom{xx}+ \int_{\T^d}\bigg(\frac{1}{\tau}I^h[(c_h-c_\delta^h)(c_h+1-\phi_h^n)]
	- c_h^nm(\phi_h^n)\big(\na\mu_h \\
    &\phantom{xx}- c_h^n\na(c_h+1-\phi_h^n)\big)
	\cdot\na(c_h+1-\phi_h^n) + g(c_h^n)|\na(c_h+1-\phi_h^n)|^2\bigg)dx \\
	&=: I_1+I_2+I_3,
\end{align*}
where we collect the gradient terms, quadratic expressions in $(\phi_h,c_h,\mu_h)$, and the nonlinear terms:
\begin{align*}
  I_1 &= \int_{\T^d}\bigg(m(\phi_h^n)\big|\na\mu_h - c_h^n\na(c_h+1-\phi_h^n)\big|^2
	+ \frac{\eps}{\tau}|\na\phi_h|^2 + g(c_h^n)|\na(c_h+1-\phi_h^n)|^2 \\
	&\phantom{xx}+\sigma\tau I^h[(\mu_h-\mu_h^n)\mu_h]\bigg)dx, \\
	I_2 &= \int_{\T^d}I^h\bigg[\bigg\{\frac{1}{\tau}(\phi_h-\phi_h^n)\mu_h
	- (\mu_h+c_h)\frac{\phi_h}{\tau}
	+ \frac{1}{\tau}(c_h-c_\delta^h)(c_h+1-\phi_h^n)\bigg\}\bigg]dx, \\
	I_3 &= \frac{1}{\eps}\int_{\R^d}I^h\bigg[(f_{1,\delta}(\phi_h)-f_2(\phi_h^n))\frac{\phi_h}{\tau}\bigg] dx.
\end{align*}
It follows from Young's inequality that
\begin{align*}
  I_1 &\ge \frac{\eps}{\tau}\int_{\T^d}|\na\phi_h|^2 dx
	+ \int_{\T^d}m(\phi_h^n)\big|\na\mu_h - c_h^n\na(c_h+1-\phi_h^n)\big|^2 dx \\
	&\phantom{xx}
	+ \int_{\T^d}g(c_h^n)|\na(c_h+1-\phi_h^n)|^2dx
    + \frac{\sigma\tau}2\int_{\T^d}I^h[|\mu_h|^2]dx
    -\frac{\sigma\tau}2\int_{\T^d}I^h[|\mu_h^n|^2]dx, \\
  I_2 &\ge \int_{\T^d}I^h\bigg[\bigg(\frac{c_h^2}{2\tau} -\frac{\sigma\tau}4|\mu_h|^2\bigg) \bigg]dx
	- \frac{C}{\sigma\tau}\int_{\T^d}I^h\bigg[\bigg(\phi_h^2
	+ \frac{1}{\tau^2}|\phi_h^n|^2 + |c_h^n|^2+1\bigg)\bigg]dx,
\end{align*}
where $C>0$ does not depend on $\delta$.
We write the remaining term $I_3$ as $I_3=\int_{\T^d}I^h[I_{31}+I_{32}]dx$, where
\begin{align*}
  I_{31} = \frac{1}{\eps\tau}(f_{1,\delta}(\phi_h)-f_2(\phi_h^n))
  (\phi_h-\phi_h^n), \quad
  I_{32} = \frac{1}{\eps\tau}(f_{1,\delta}(\phi_h)-f_2(\phi_h^n))\phi_h^n.
\end{align*}
We use the properties
\begin{equation*}
\begin{aligned}
  |f_{1,\delta}(\phi_h)|&\le C\delta^{-1} (1+|\phi_h|), &\quad
  \quad |f_2(\phi_\delta^h)|&\le C(1+|\phi_h^n|), \\
  F_\delta(\phi_h^n) &\le C\delta^{-1}(1+|\phi_h^n|^2), &\quad
  F_\delta(\phi_h)&\ge C\delta^{-1}\phi_h^2 - C,
\end{aligned}
\end{equation*}
where $C>0$ does not depend on $\delta$ and $F_\delta:=F_{1,\delta}-F_2$.
Then, by the convexity of $f_{1,\delta}$ and $f_2$,
\begin{align*}
  I_{31} \ge \frac{1}{\eps\tau}(F_\delta(\phi_h)-F_\delta(\phi_h^n))
	\ge \frac{C}{\delta \eps\tau}\phi_h^2 - \frac{C}{\delta\eps\tau}(1+|\phi_h^n|^2).
\end{align*}
Finally, we deduce from
\begin{align*}
  I_{32} &\ge \frac{1}{\eps\tau}f_{1,\delta}(\phi_h)\phi_h^n
  - \frac{C}{\eps\tau}((\phi_h^n)^2+1)
	\ge -\frac{C}{\eps\tau}\left(\phi_h^2 +\left(1+\frac1{\delta^2}\right) |\phi_h^n|^2 + 1\right),
\end{align*}
where $C>0$ does not depend on $\delta$, that for sufficiently small $\delta>0$,
\begin{align*}
  I_3	\ge \frac{C}{\eps\tau}\left(\frac1{\delta}-1\right)\int_{\T^d}\phi_h^2 dx
	-\frac{ C}{\eps\tau}\int_{\T^d}\bigg(
	\bigg(1+\frac1\delta+\frac1{\delta^2}\bigg)|\phi_h^n|^2 + 1\bigg)dx.
\end{align*}
Summarizing the previous estimates, we find that
\begin{align*}
  \langle & S(\phi_h,c_h,\mu_h),(\zeta_h,\xi_h,\chi_h)\rangle
	\ge \int_{\T^d}\bigg\{\frac{\eps}{2\tau}|\na\phi_h|^2
	+ m(\phi_h^n)\big|\na\mu_h-c_h^n
	\na(c_h+1-\phi_h^n)\big|^2 \\
	&\phantom{xx} + g(c_h^n)|\na(c_h+1-\phi_h^n)|^2
	+I^h\bigg[\bigg(\frac{C}{\eps\tau}\bigg(\frac1{\delta}-1\bigg)
	- \frac{C}{\sigma\tau}\bigg)\phi_h^2 + \frac{c_h^2}{2\tau} 
	+ \frac{\sigma\tau}4\mu_h^2-C^n\bigg]\bigg\} dx,
\end{align*}
where $C^n>0$ depends on $(\phi_h^n,c_h^n,\mu_h^n)$, $\delta$, $\sigma$, and $\tau$ but not on $(\phi_h,c_h,\mu_h)$. For any fixed $\sigma>0$ and $\tau>0$, choosing $\delta$ sufficiently small, we have
$$
  \langle S(\phi_h,c_h,\mu_h),(\zeta_h,\xi_h,\chi_h)\rangle
	\ge C\|(\phi_h,c_h,\mu_h)\|_2^2 - C^n.
$$
Since the norms $\|L(\cdot)\|_2$ and $\|\cdot\|_2$ on a finite-dimensional space are equivalent, i.e.\
\begin{align*}
	C''\|{L}(\phi_h,c_h,\mu_h)\|_2\ge \|(\phi_h,c_h,\mu_h)\|_2  
  \ge C'\|{L}(\phi_h,c_h,\mu_h)\|_2= C'\|(\zeta_h,\xi_h,\chi_h)\|_2
   = C'R
\end{align*}
for some positive constants $C'$ and $C^{''}$ independent of $\delta$,
 for sufficiently large $R>0$,
\begin{align}\label{3.pos}
  \langle S(\phi_h,c_h,\mu_h),(\mu_h,c_h,\phi_h/\tau)\rangle
  &= \langle S(\phi_h,c_h,\mu_h),(\zeta_h,\xi_h,\chi_h)\rangle \\
	&\ge C {R^2}-C^n > 0. \nonumber
\end{align}
On the other hand, since $(\zeta_h,\xi_h,\chi_h)=L(\phi_h,c_h,\mu_h)\in B_R$ is a fixed point of $G_R$, we have
$$
  G_R(L(\phi_h,c_h,\mu_h)) = -R\frac{S(\phi_h,c_h,\mu_h)}{\|S(\phi_h,c_h,\mu_h)\|_2}
$$
and therefore,
\begin{align*}
  \langle &S(\phi_h,c_h,\mu_h),(\mu_h,c_h,\phi_h/\tau)\rangle \\
	&= -\frac{1}{R}\|S(\phi_h,c_h,\mu_h)\|_2\big\langle G_R(L(\phi_h,c_h,\mu_h)),
	(\mu_h,c_h,\phi_h/\tau)\big\rangle \\
	&= -\frac{1}{R}\|S(\phi_h,c_h,\mu_h)\|_2\big\langle L(\phi_h,c_h,\mu_h),
	(\mu_h,c_h,\phi_h/\tau)\big\rangle \\
	&= -\frac{1}{R}\|S(\phi_h,c_h,\mu_h)\|_2\big\langle (\mu_h,c_h,\phi_h/\tau),
	(\mu_h,c_h,\phi_h/\tau)\big\rangle \\
	&= -\frac{1}{R}\|S(\phi_h,c_h,\mu_h)\|_2\int_{\mathbb{T}^d}I^h[\mu_h^2+c_h^2+\phi_h^2/\tau^2]dx \le 0.
\end{align*}
This inequality contradicts \eqref{3.pos}. We conclude that, if $R>0$ is sufficiently large, the mapping $S\circ L^{-1}$ has a zero in $B_R$, which guarantees the existence of a solution to the regularized scheme.
\end{proof}

\subsection{Uniform bounds}

The energy inequality in Lemma \ref{lem.Edelta} implies that $(\na\phi_\delta^{n+1})_\delta$ is bounded in $L^2(\T^d)$ and $(\phi_\delta^{n+1})_\delta$ is bounded in $L^1(\T^d)$. By the Poincar\'e--Wirtinger inequality, $(\phi_\delta^{n+1})_\delta$ is bounded in $H^1(\T^d)$. Furthermore, the energy inequality provides a bound for $(c_\delta^{n+1})_\delta$ in $L^2(\T^d)$ and, because of the stabilization, of $(\mu_\delta^{n+1})_\delta$ in $L^2(\T^d)$. 

For the limit $\delta\to 0$, we need gradient bounds for $(\mu_\delta^{n+1})_\delta$, $(c_\delta^{n+1})_\delta$ and integral bound for $I^h[|f_{1,\delta}(\phi_\delta^{n+1})|^2]$, which do not directly follow from our approach because of the degeneracies and the linearized scheme. Therefore, for the first and the second terms, we use the inverse inequality \cite[(2.4),(2.5)]{BBH2000}
\begin{align}\label{4.prop}
  \|v\|_{L^2(\T^d)} \le Ch^{-d/2}\|v\|_{L^1(\T^d)}, \quad
  \|v_h\|_{H^1(\T^d)} \le Ch^{-1}\|v_h\|_{L^2(\T^d)}
  \quad\mbox{for }v_h\in X_h,
\end{align}
which yields $\delta$-uniform bounds for $(\mu_\delta^{n+1})$, $(c_\delta^{n+1})$ in $H^1(\T^d)$. 

{It remains to establish a uniform bound for $\int_{\T^d}I^h[|f_{1,\delta}(\phi_\delta^{n+1})|^2]dx$ with respect to $\delta$. Choosing $\chi_h=I^h[f_{1,\delta}(\phi_\delta^{n+1})]$ in \eqref{3.mu} and rearranging the terms shows that
\begin{align*}
  \int_{\T^d}&I^h[f_{1,\delta}(\phi_\delta^{n+1})I^h[f_{1,\delta}(\phi_\delta^{n+1})]]dx
	= -\eps\int_{\T^d}\na\phi_\delta^{n+1}\cdot\na I^h[f_{1,\delta}(\phi_\delta^{n+1})]dx\\
	&+ \eps\int_{\T^d}I^h[\mu_\delta^{n+1}I^h[f_{1,\delta}(\phi_\delta^{n+1})]]dx + \int_{\T^d}I^h[(f_2(\phi_\delta^n)+c_\delta^{n+1})
	I^h[f_{1,\delta}(\phi_\delta^{n+1})]]dx.
\end{align*}
The left-hand side can be rewritten, leading to
\begin{align}
  \int_{\T^d}&I^h[|f_{1,\delta}(\phi_\delta^{n+1})|^2]dx
	= -\eps\int_{\T^d}\na\phi_\delta^{n+1}\cdot\na I^h[f_{1,\delta}(\phi_\delta^{n+1})]dx \nonumber\\
	&+ \eps\int_{\T^d}I^h[\mu_\delta^{n+1}f_{1,\delta}(\phi_\delta^{n+1})]dx + \int_{\T^d}I^h[(f_2(\phi_\delta^n)+c_\delta^{n+1})
	f_{1,\delta}(\phi_\delta^{n+1})]dx.\label{f-d-1}
\end{align}}

{The discrete energy estimate and definition \eqref{func.f} imply a lower and upper bound for $\phi_\delta$, yielding
\begin{equation*}
  \int_{\mathbb{T}^d}I_h[(\phi_\delta^{n+1})_-^2]dx
  +\int_{\mathbb{T}^d}I_h[(\phi_\delta^{n+1}-1)_+^2]dx\le C\delta,
\end{equation*}
which, by the inverse inequalities \eqref{4.prop}, shows that
\begin{equation*}
  \|(\phi_\delta^{n+1})_-\|_{L^\infty(\T^d)}
  +\|(\phi_\delta^{n+1}-1)_+\|_{L^\infty(\T^d)} \le C\delta^{1/2}.
\end{equation*}
We deduce from this bound as well as from definitions \eqref{lem.Edelta} and \eqref{func.f} that
\begin{equation}\label{prop.phi.2}
	\| F_{1,\delta}(\phi_\delta^{n+1})\|_{L^\infty(\T^d)}\le C.
\end{equation}}

{Now we prove (similar to \cite{BBH2000,JFC1975}) the uniform estimate
\begin{align}\label{con.2}
	-\eps\int_{\mathbb{T}^d}\na\phi_\delta^{n+1}\cdot\na I^h[f_{1,\delta}(\phi_\delta^{n+1})]dx\le C.
\end{align}
It follows from \eqref{prop.phi.2} for any simplex $\kappa\in\mathcal{T}^h$ that
\begin{align*}
	-\eps\int_{\kappa}&\na\phi_\delta^{n+1}\cdot\na I^h[f_{1,\delta}(\phi_\delta^{n+1})]dx\\
	&=-\eps\sum_{i,j\in J}\phi_\delta^{n+1}(x_i)f_{1,\delta}(\phi_\delta^{n+1}(x_j))\int_{\kappa}\na p_i(x)\cdot\na p_j(x)dx\\
	&=-\eps\sum_{i,j\in J}(\phi_\delta^{n+1}(x_i)-\phi_\delta^{n+1}(x_j))f_{1,\delta}(\phi_\delta^{n+1}(x_j))\int_{\kappa}\na p_i(x)\cdot\na p_j(x)dx\\
	&\le-\eps\sum_{i,j\in J}(F_{1,\delta}(\phi_\delta^{n+1}(x_i))-F_{1,\delta}(\phi_\delta^{n+1}(x_j)))\int_{\kappa}\na p_i(x)\cdot\na p_j(x)dx \le C.
\end{align*}
Here, we use \eqref{prop.phi.2} and the convexity of $F_{1,\delta}$.
This shows \eqref{con.2} after summing over $\kappa\in\mathcal{T}^h$.}

{Thus, according to the uniform $L^2(\T^d)$ bound for $(\phi_\delta^n, c_\delta^{n+1}, \mu_\delta^{n+1})_\delta$, definition $f_2(\phi)=\theta_0\phi-\theta_0/2$, and Young's inequality, it follows from \eqref{f-d-1} and \eqref{con.2} that
\begin{align}
  \int_{\T^d}I^h[|f_{1,\delta}(\phi_\delta^{n+1})|^2]dx\le C.\label{f-d-2}
\end{align}}

\subsection{Limit $\delta\to 0$}

We have shown that $(\phi_\delta^{n+1})$, $(c_\delta^{n+1})$, and $(\mu_\delta^{n+1})$ are bounded in $H^1({\T^d})$ (with respect to $\delta$). Hence, there exist subsequences which are not relabeled such that, as $\delta\to 0$,
\begin{align*}
  & \na\phi_\delta^{n+1}\rightharpoonup \na\phi_h^{n+1},
  \quad\na c_\delta^{n+1}\rightharpoonup\na c_h^{n+1},
  \quad\na\mu_\delta^{n+1}\rightharpoonup\na\mu_h^{n+1} \quad
  \mbox{weakly in }L^2(\T^d).
\end{align*} 
Since $h_c(\phi,c)=c+1-\phi$ and $h_\phi(\phi,c)=-c$ are linear, we have
\begin{align*}
  \na h_c(\phi_h^n,c_\delta^{n+1})\rightharpoonup\na h_c(\phi_h^n,c_h^n),
  \quad h_\phi(\phi_h^n,c_\delta^{n+1})\rightharpoonup h_\phi(\phi_h^n,c_h^n)
  \quad\mbox{weakly in }L^2(\T^d).
\end{align*}
(In fact, there is a subsequence such that $(h_\phi(\phi_h^n,c_\delta^{n+1}))$ converges even strongly in $L^2(\T^d)$.) These convergences are sufficient to pass to the limit $\delta\to 0$ in the linear equations \eqref{3.phi}--\eqref{3.c}. The difficult part is the limit in the remaining equation \eqref{3.mu} with the nonlinear term $f_{1,\delta}(\phi_\delta^{n+1})$.

{By \eqref{f-d-2}, there exists $g_h^{n+1}(x)\in X_h$ such that $I^h[f_{1,\delta}(\phi_\delta^{n+1})](x)\to g_h^{n+1}(x)$ on $\T^d$ as $\delta\to 0$ and $f_{1,\delta}(\phi_\delta^{n+1}(x_j))$ is uniformly bounded in $\delta$. Also, using $\phi_\delta^{n+1}(x_j)\to \phi_h^{n+1}(x_j)$ for all $j\in J$ as $\delta\to 0$,  and the fact that, for all $s\in (-\infty,\infty)$, $[f_{1,\delta}]^{-1}(s)\to [f_{1}]^{-1}(s)$ as $\delta\to 0$, where $\phi_\delta^{n+1}=[f_{1,\delta}]^{-1}(s)$ and $\phi_h^{n+1}=[f_{1}]^{-1}(s)$ are the inverse functions of $s=f_{1,\delta}(\phi_\delta^{n+1}):(-\infty,\infty)\mapsto (-\infty,\infty)$ and $s=f_1(\phi_h^{n+1}):(0,1)\mapsto (-\infty,\infty)$, we have $\phi^{n+1}_h(x_j)=[f_{1}]^{-1}(g_h^{n+1}(x_j))$ and therefore $g_h^{n+1}(x_j)=f_1(\phi_h^{n+1}(x_j))$ for all $j\in J$. Hence, we have that 
\begin{align}
	\int_{\T^d}&I^h[|f_{1}(\phi_h^{n+1})|^2]dx\le C, \label{f-d-3}
\end{align}
which immediately implies that  $0<\phi_h^{n+1}(x_j)<1$ for all $j\in J$ and $0<\phi^{n+1}_h(x)<1$ on $\T^d$.} 
Consequently, for any $\chi_h\in X_h\subset L^2(\T^d)$, when $\delta\to 0$,  $$ \int_{\T^d}I^h[f_{1,\delta}(\phi_\delta^{n+1})\chi_h] dx \to\int_{\T^d} I^h[f_1(\phi_h^{n+1})\chi_h]dx.$$



The previous convergences are sufficient to pass to the limit $\delta\to 0$ in \eqref{3.mu}, and we conclude that $(\phi_h^{n+1},c_h^{n+1},\mu_h^{n+1})$ is a weak solution to \eqref{2.phi}--\eqref{2.mu}. Moreover, by the lower semicontinuity of the gradient $L^2({\T^d})$ norm, we can perform the limit $\delta\to 0$ in the approximate energy inequality in Lemma \ref{lem.Edelta}, which gives \eqref{2.Eineq} and finishes the proof of Theorem \ref{thm.ex}.


\section{Numerical experiments}\label{sec.num}

We illustrate the dynamical behavior of the solutions to our model.
To this end, we solve the regularized scheme \eqref{3.phi}--\eqref{3.mu}, replacing $(\phi_h^n,c_h^n)$ by $(\phi_\delta^n,c_\delta^n)$, implemented in FreeFem \cite{Hec12}. Equation \eqref{3.mu} is solved by using the Newton method. It is possible to use the unregularized system \eqref{2.phi}--\eqref{2.mu} if the initial data lies in the interval $(\eta,1-\eta)$ for some small $\eta>0$, and the numerical results for $\delta=0$ and $\delta=10^{-3}$ are basically the same. The spatial domain is the rectangle $\T^2=(0,2\pi)\times(0,2\pi)$, discretized by a uniform triangular mesh of $60\times 60$ triangles. The physical and numerical parameters are 
$$
  \delta = 10^{-3}, \quad \eps=0.15, \quad \sigma = 0.1, \quad
  \theta_0 = 7, \quad \tau = 10^{-3},
$$ 
if not stated otherwise. Chen et al.\ \cite{CWWW19} have used the value $\theta_0=3$ for the Cahn--Hilliard equation with a Flory--Huggins-type potential, which motivates our choice of $\theta_0$. Furthermore, we choose the function $g(c)=c^2$, which satisfies the conditions stated in Theorem \ref{thm.ex}.

\subsection{Convergence, energy inequality, and physical bounds}\label{sec.num1}

In the first example, we compute numerically the temporal convergence rate using the initial data $\phi^0(x) = 0.08\,\mbox{rand}+0.2$ and $c^0(x) = 0.1\,\mbox{rand} + 0.4$, where rand yields a uniformly distributed random number in the interval $[0,1]$ (we have used the command {\tt randreal1} in FreeFem); see Figure \ref{fig.init}. 
The reference solution is computed with the time step size $\tau=10^{-4}$. The numerical errors in the $L^2(\T^2)$ norm (more precisely, we used the command {\tt int2D} of FreeFem) at time $T=0.064$ are shown in Table \ref{table}. We confirm the first-order convergence rate when the time step sizes are not too large.

\begin{figure}[ht]
	\includegraphics[width=60mm]{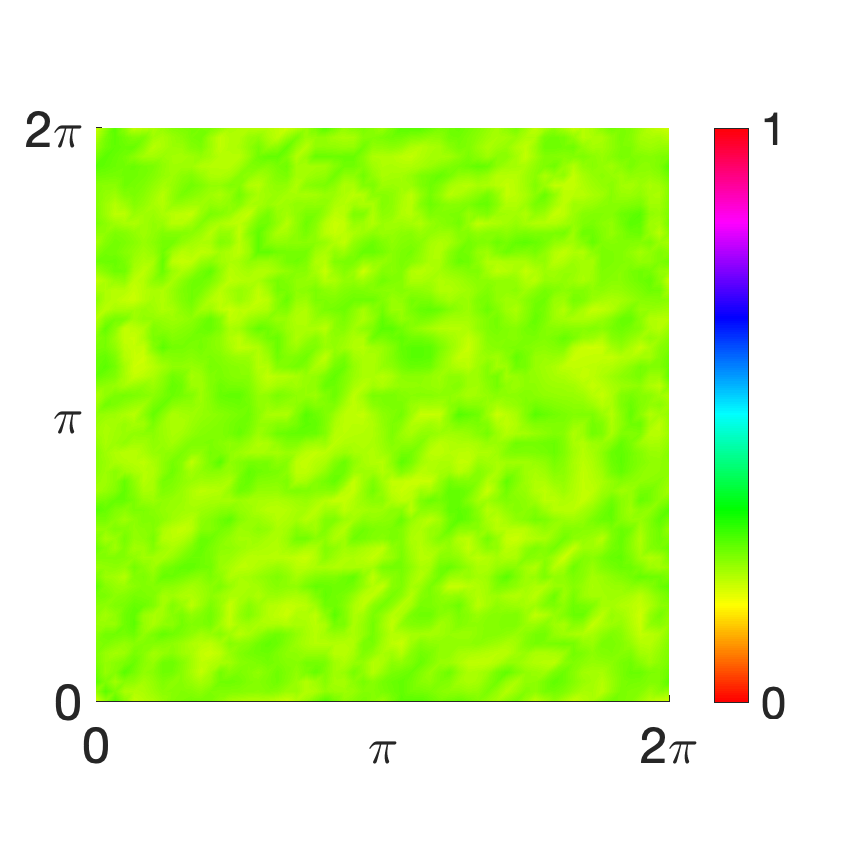}
	\includegraphics[width=60mm]{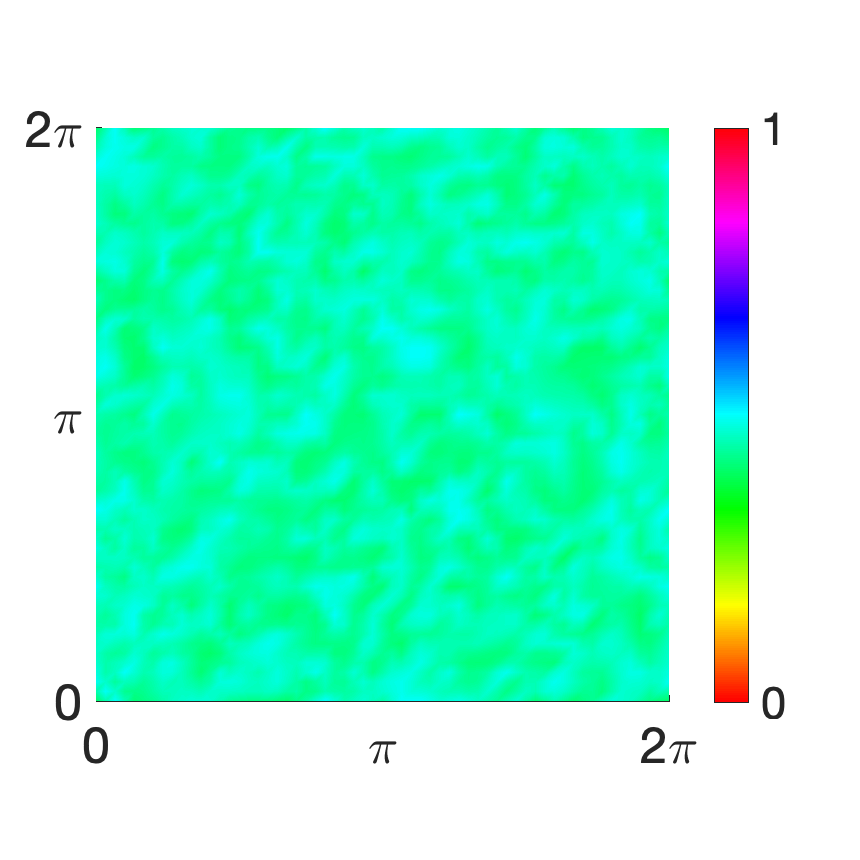}
	\caption{Initial data $\phi^0$ (left) and $c^0$ (right) used in Section \ref{sec.num1}.}
	\label{fig.init}
\end{figure}

\begin{table}
\begin{tabular}{p{20mm}cccc}
\hline
time step & $\|\phi-\phi_{\rm ref}\|_{L^2}$ & rate & $\|c-c_{\rm ref}\|_{L^2}$ & rate \\
\hline
0.0064 & 1.15e-02 & & 1.24e-02 &  \\
0.0032 & 6.31e-03 & 0.86 & 6.52e-03 & 0.92 \\
0.0016 & 3.25e-03 & 0.95 & 3.26e-03 & 1.00 \\
0.0008 & 1.57e-03 & 1.05 & 1.55e-03 & 1.08 \\
0.0004 & 6.86e-04 & 1.19 & 6.68e-04 & 1.21 \\
0.0002 & 2.31e-04 & 1.57 & 2.23e-04 & 1.58 \\
\hline
\end{tabular}
\caption{Temporal convergence rate.}
\label{table}
\end{table}

We discuss the minimal values of the numerical concentrations $c_h$ for various spatial mesh sizes. We choose a small perturbation for the initial data, namely $\phi^0(x) = 0.08\,\mbox{rand}+0.2$ and $c^0(x) = 0.001\,\mbox{rand}$. Recall that rand yields random numbers from the interval $[0,1]$, so the initial concentration is nonnegative. The numerical solutions are solved with time step size $\tau = 10^{-4}$ and on three uniform triangular meshes of $30\times 30$, $60\times 60$, and $90\times 90$ triangles. The minimal values of $c_h$ are presented in Table \ref{table2}. We observe that in principle the numerical concentration becomes positive for sufficiently small mesh sizes. Since $c^0\ge 0$ only, the analytical results of \cite{JuUn05} predict that $c_h\ge -Ch^\alpha$ for some $\alpha>1$. This explains the negative values in the last row of Table \ref{table2}. Thus, although our scheme does not preserve the nonnegativity of the concentrations in general, the minimal values of $c_h$ are under control. In applications, the solute concentration is usually not close to zero such that we do not expect negative values for $c_h$ in the corresponding numerical simulations. The test presented here illustrates an extreme case.

\begin{table}
	\begin{tabular}{p{20mm}rrr}
	\hline
	time/mesh & $30\times30$ & $60\times60$ & $90\times90$\\
	\hline
	$t = 0.2$ & $-1.44$e-05 &  1.29e-06 & 1.68e-06 \\
	$t = 0.4$ & $-4.02$e-05 & $-1.21$e-09 & 1.55e-06 \\
	$t = 0.6$ & $-1.41$e-04 & $-8.53$e-06 & $-1.49$e-06 \\
	\hline
	\end{tabular}
	\caption{Minimal values of $c_h$ at various times and for various spatial mesh sizes.}
	\label{table2}
\end{table}

Next, we illustrate how the energy decay and extreme values of $\phi$ are influenced by variations of the parameter $\delta$. The initial data is $\phi^0(x) = 0.08\,\mbox{rand}+0.2$, $ c^0(x) = 0.1\,\mbox{rand} + 0.4$. Figure \ref{fig.E.tau} (left) shows that the discrete energy values for various time steps are decreasing, however, not with a uniform rate. The extremal values $\min\phi_h^n$ and $\max\phi_h^n$ at times $t=0,\ldots,1$ in Figure \ref{fig.E.tau} (right) preserve the physical bounds for the fiber phase. In Figure \ref{fig.E.delta}, we present the energy decay and extremal values of $\phi_h^n$ for various values of $\delta$. The initial data is $\phi^0(x) = 0.6\,\mbox{rand}+0.2$, $ c^0(x) = 0.1\,\mbox{rand} + 0.4$.
The scheme is energy stable for all values of $\delta$, but the extremal values of $\phi_h^n$ lie in the interval $(0,1)$ only if $\delta$ is sufficiently small. If $\delta=0.1$, the extremal values become larger than one.

\begin{figure}[ht]
\includegraphics[width=60mm]{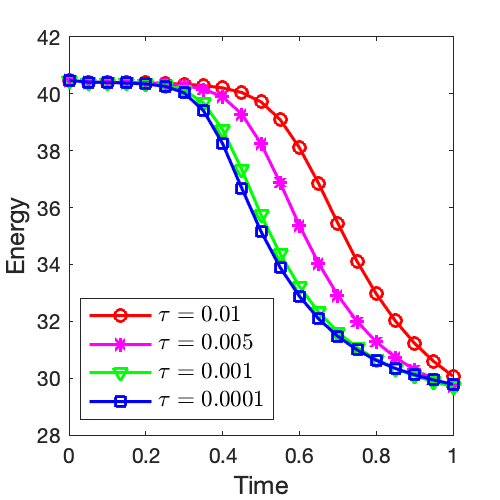}
\includegraphics[width=60mm]{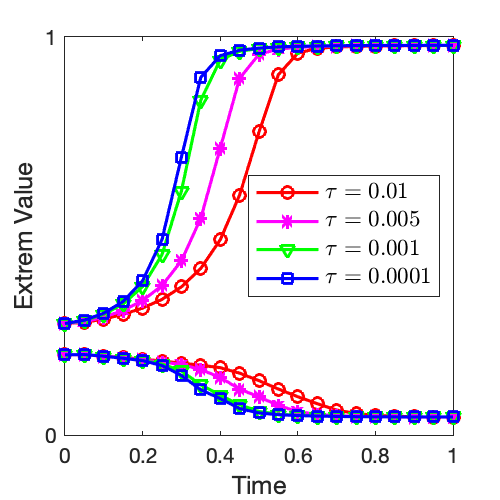}
\caption{Discrete energy (left) and extreme values $\min\phi_h^n$ and $\max\phi_h^n$ (right) for various values of the time step $\tau$.}
\label{fig.E.tau}
\end{figure}

\begin{figure}[ht]
\includegraphics[width=60mm]{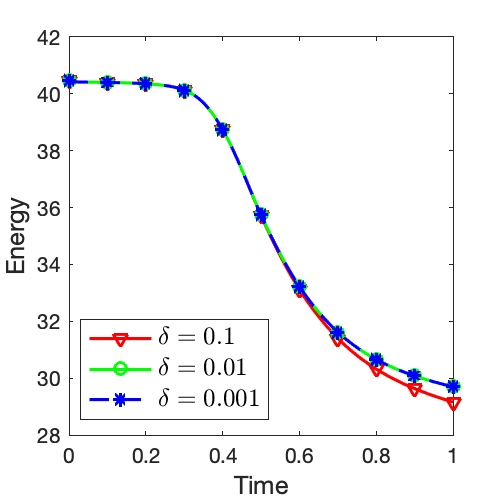}
\includegraphics[width=60mm]{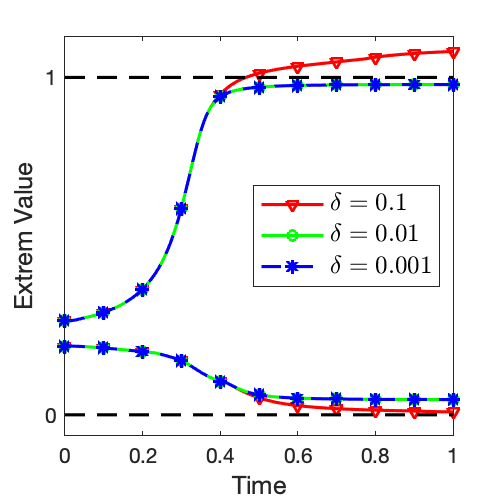}
\caption{Discrete energy (left) and extreme values $\min\phi_h^n$ and $\max\phi_h^n$ (right) for various values of the regularization parameter $\delta$.}
\label{fig.E.delta}
\end{figure}

We also illustrate the influence of $\theta_0$ on the dynamics in Figure \ref{fig.E.theta}. For the chosen parameters, the physical bounds $0\le\phi_\delta^n\le 1$ and $c_\delta^n\ge 0$ are satisfied. 

\begin{figure}[ht]
	\includegraphics[width=60mm]{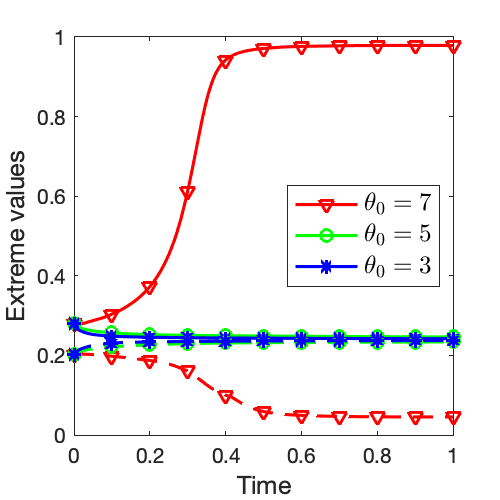}
	\includegraphics[width=60mm]{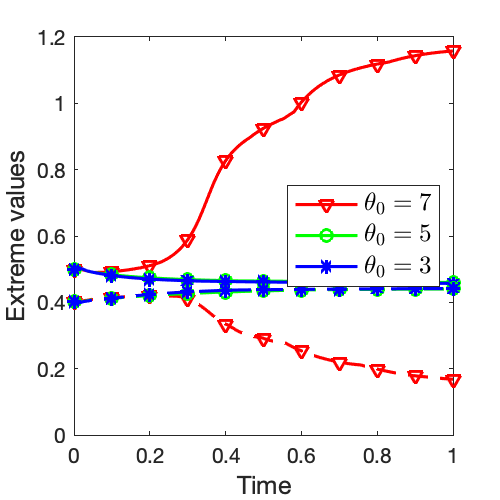}
	\caption{Extreme values $\min\phi_h^n$ and $\max\phi_h^n$  (left) and extreme values $\min c_h^n$ and $\max c_h^n$
	(right) for various values of $\theta_0$ with $\delta=10^{-3}$.}
	\label{fig.E.theta}
\end{figure}

\subsection{Phase separation}

We present simulations of phase separation in $\T^2$ with the initial data $\phi^0(x) = 0.08\,\mbox{rand}+0.2$, $ c^0(x) = 0.1\,\mbox{rand} + 0.4$ and the time step size $\tau= 10^{-3}$. Recall that $\delta=10^{-3}$, $\eps=0.15$, and $\theta_0=7$. 
Snapshots for the fiber phase $\phi$ and nutrient concentration $c$ are shown in Figures \ref{fig.phi} and \ref{fig.c}. It was shown in \cite[Section 4]{RoFo08} that a reduced equation, derived from a long-wavelength reduction, shows a hexagonal-like pattern, which is confirmed by our numerical experiments for the full model.

\begin{figure}[ht]
	\includegraphics[width=52mm]{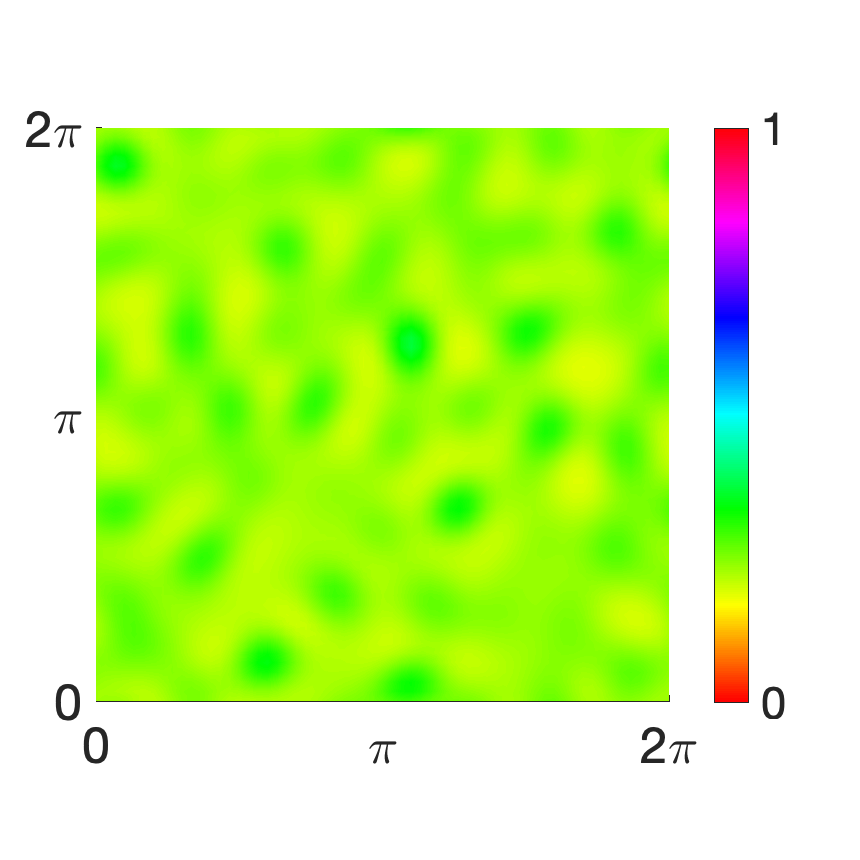}
	\includegraphics[width=52mm]{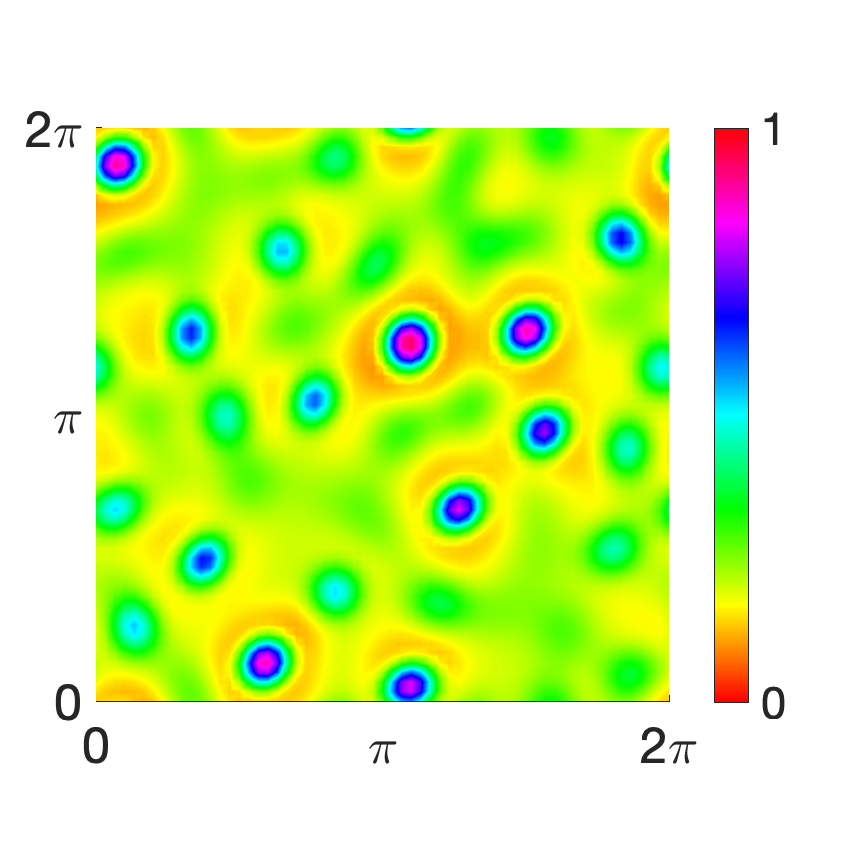}
	\includegraphics[width=52mm]{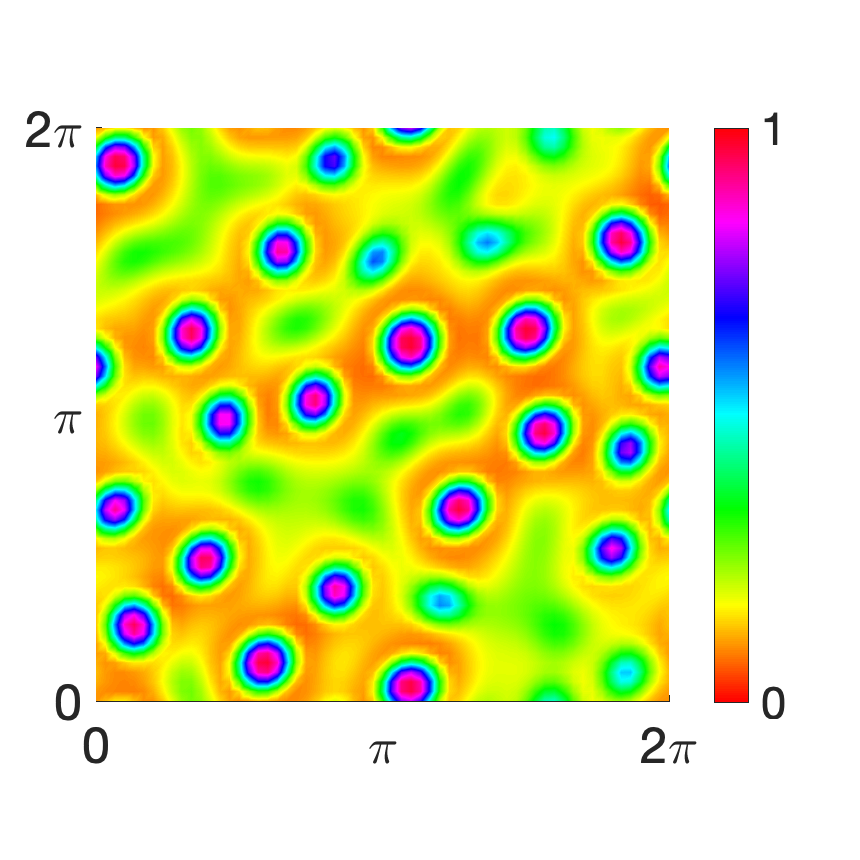}\\[-5mm]
	\includegraphics[width=52mm]{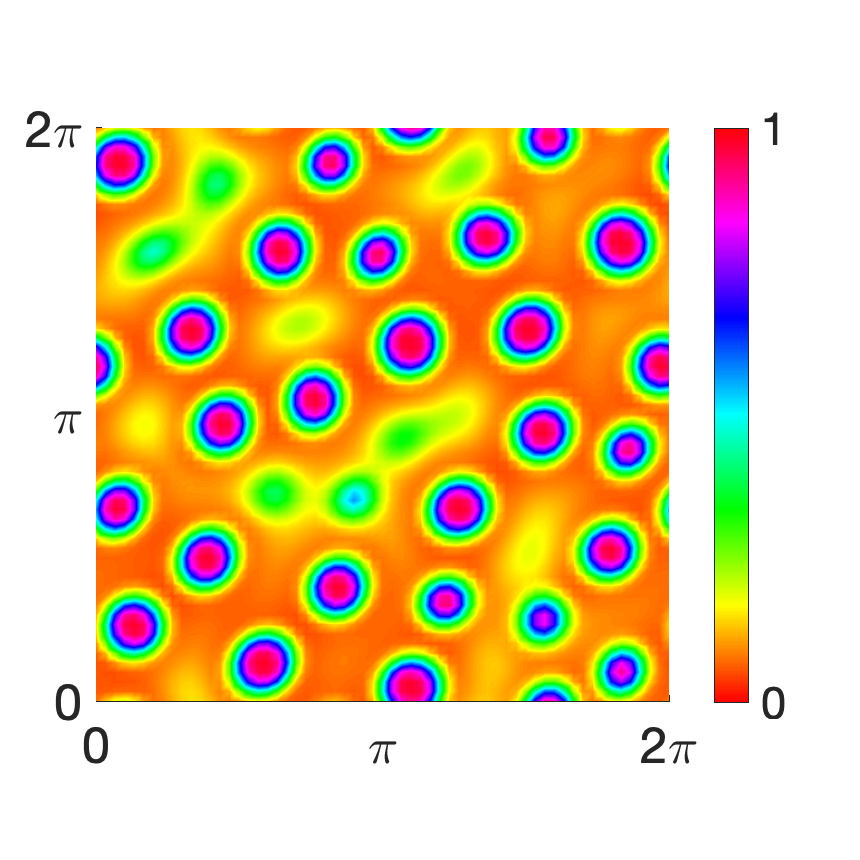}
	\includegraphics[width=52mm]{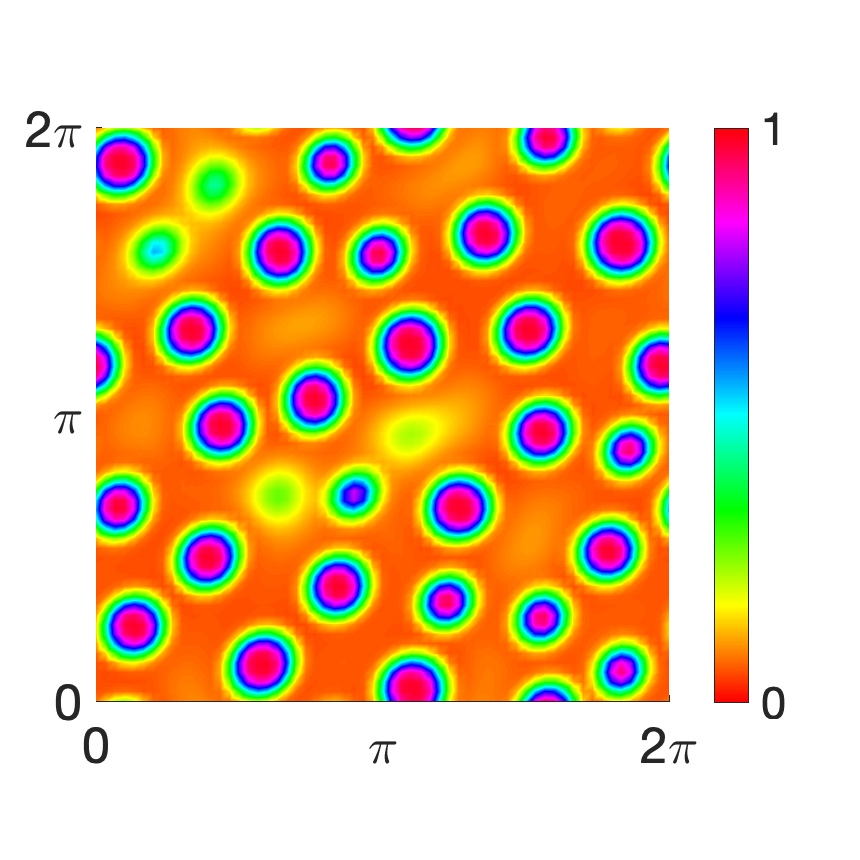}
	\includegraphics[width=52mm]{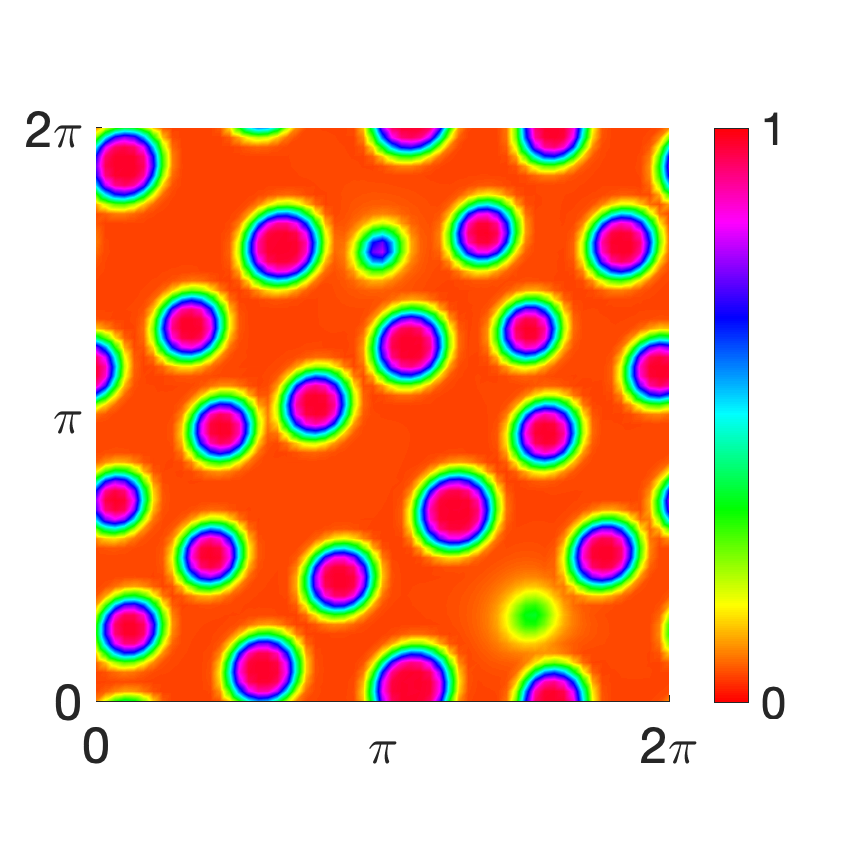}
	\caption{Fiber phase fraction $\phi$ at various times $t = 0.2,\, 0.4,\, 0.5,\,0.7,\,0.9,\,5$ with $\eps=0.15$.}
	\label{fig.phi}
	\end{figure}
	
	\begin{figure}[ht]
	\includegraphics[width=52mm]{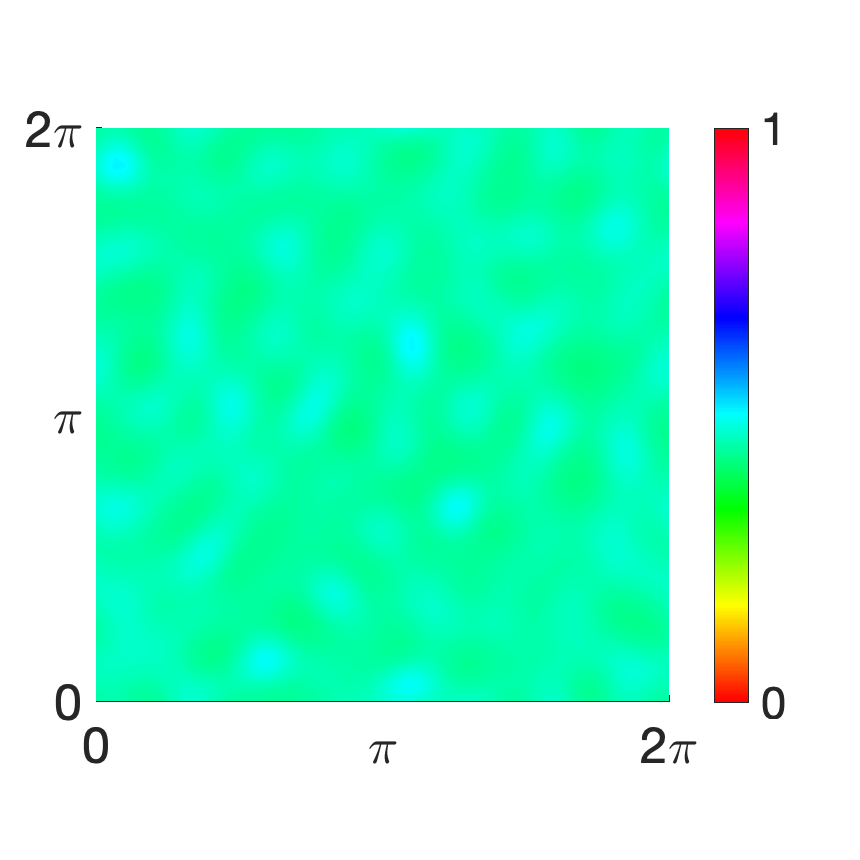}
	\includegraphics[width=52mm]{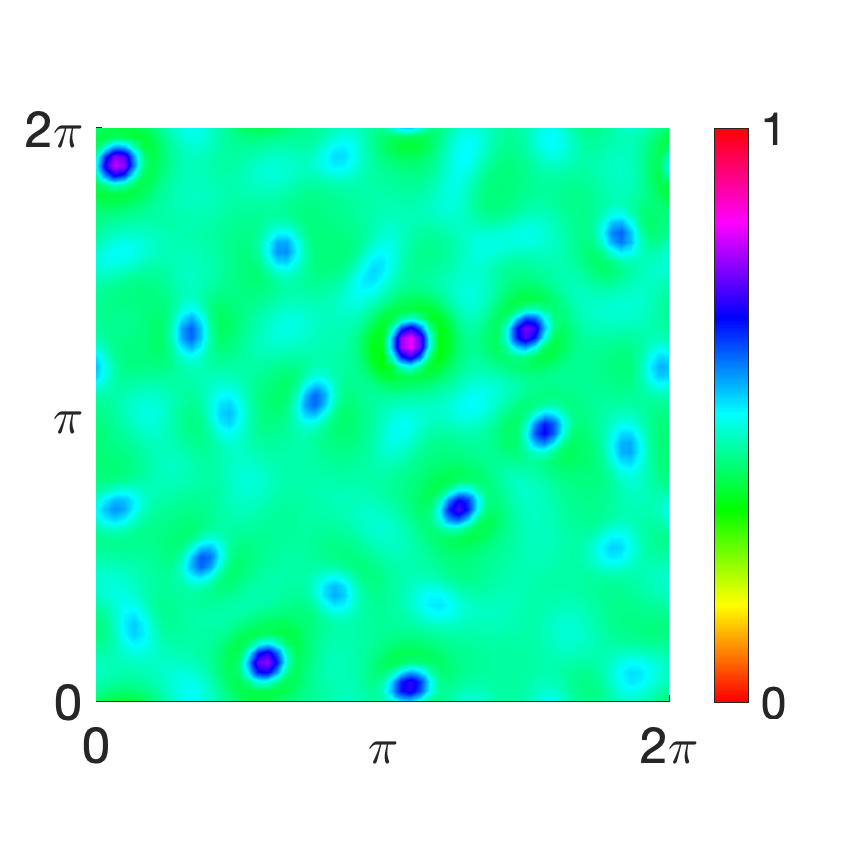}
	\includegraphics[width=52mm]{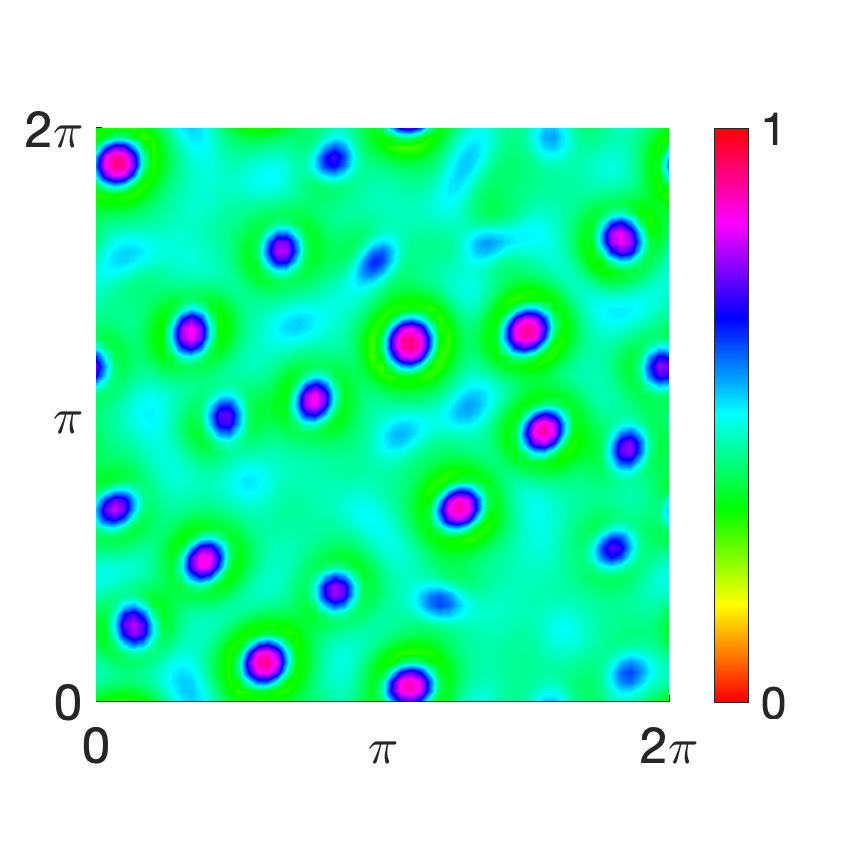}\\[-5mm]
	\includegraphics[width=52mm]{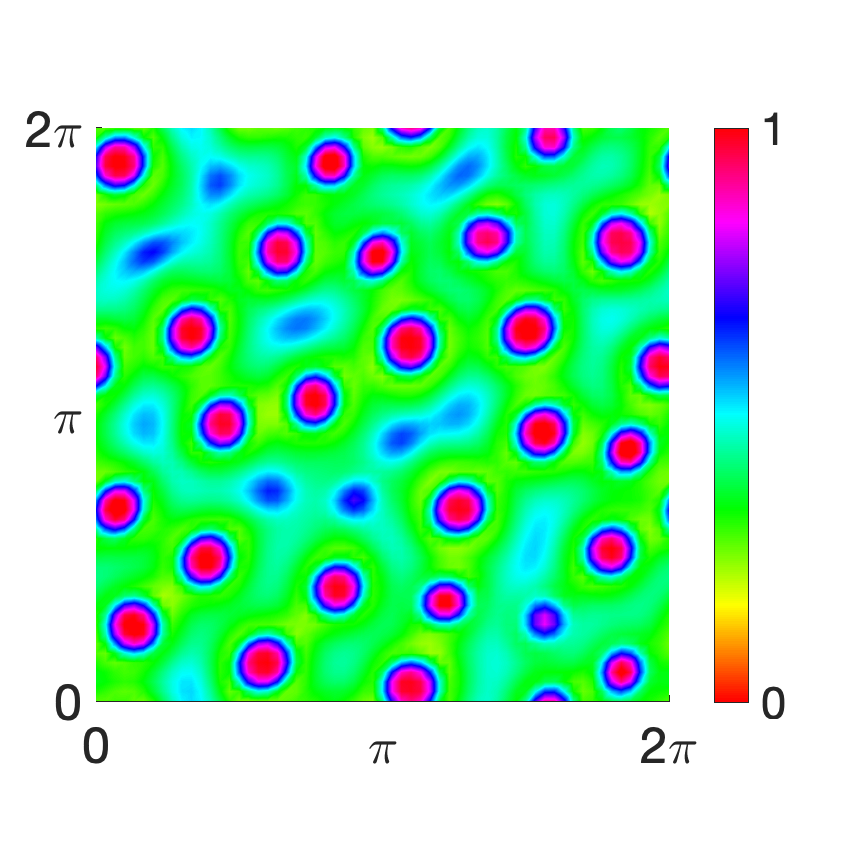}
	\includegraphics[width=52mm]{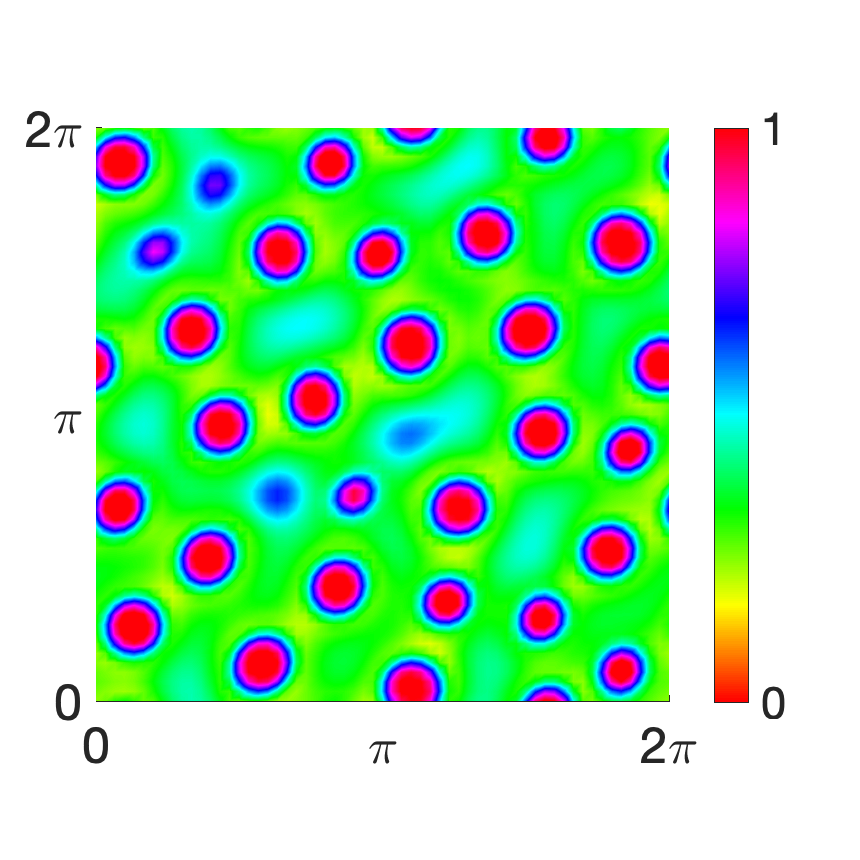}
	\includegraphics[width=52mm]{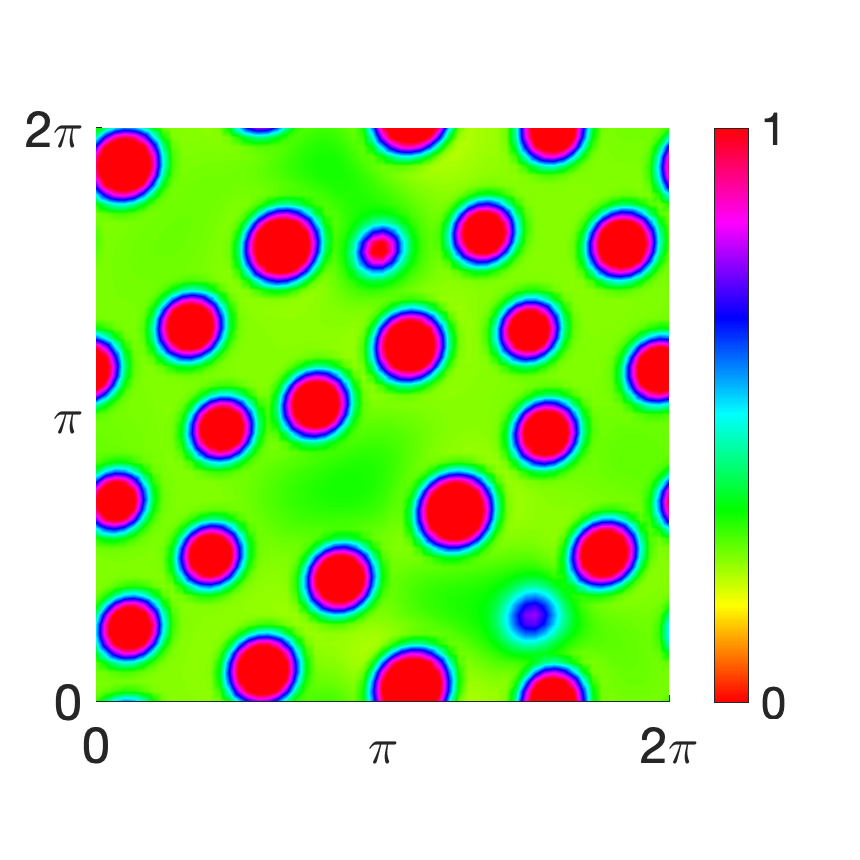}
	\caption{Nutrient concentration $c$ at various times $t = 0.2,\, 0.4,\, 0.5,\,0.7,\,0.9,\,5$ with $\eps=0.15$.}
	\label{fig.c}
\end{figure}

We have verified that the total nutrient mass is conserved, the discrete energy is decreasing, and the nutrient concentration is positive for the simulated times (not shown). This confirms the structure-preserving properties of the semi-convex-splitting scheme.

For comparison, we present simulations of phase separation when $\eps = 0.3$ is larger than in the previous example (the other parameters are unchanged). We show in Figure \ref{fig.phi.2} and \ref{fig.c.2} the snapshots of the solutions. Again we observe phase separation but the pattern coarsens. 

\begin{figure}[ht]
	\includegraphics[width=52mm]{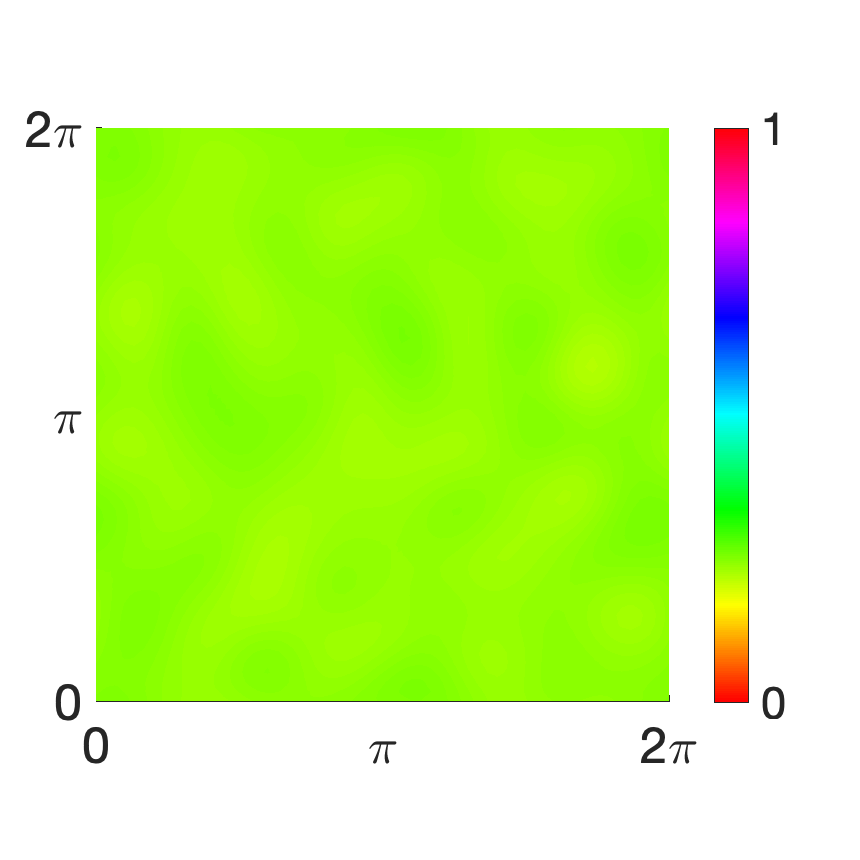}
	\includegraphics[width=52mm]{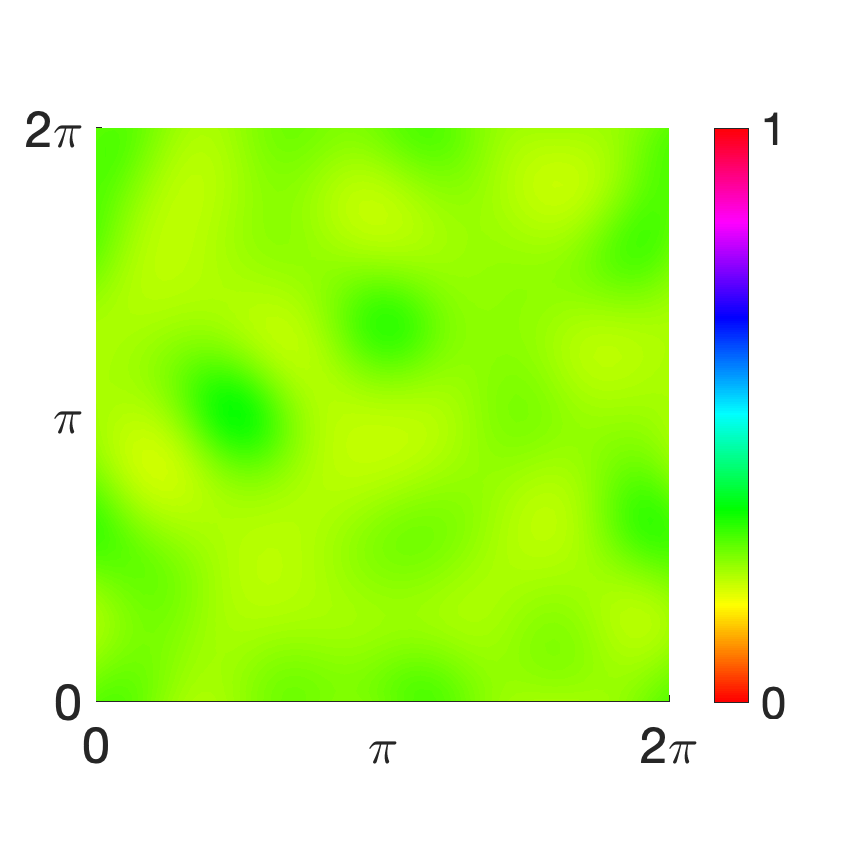}
	\includegraphics[width=52mm]{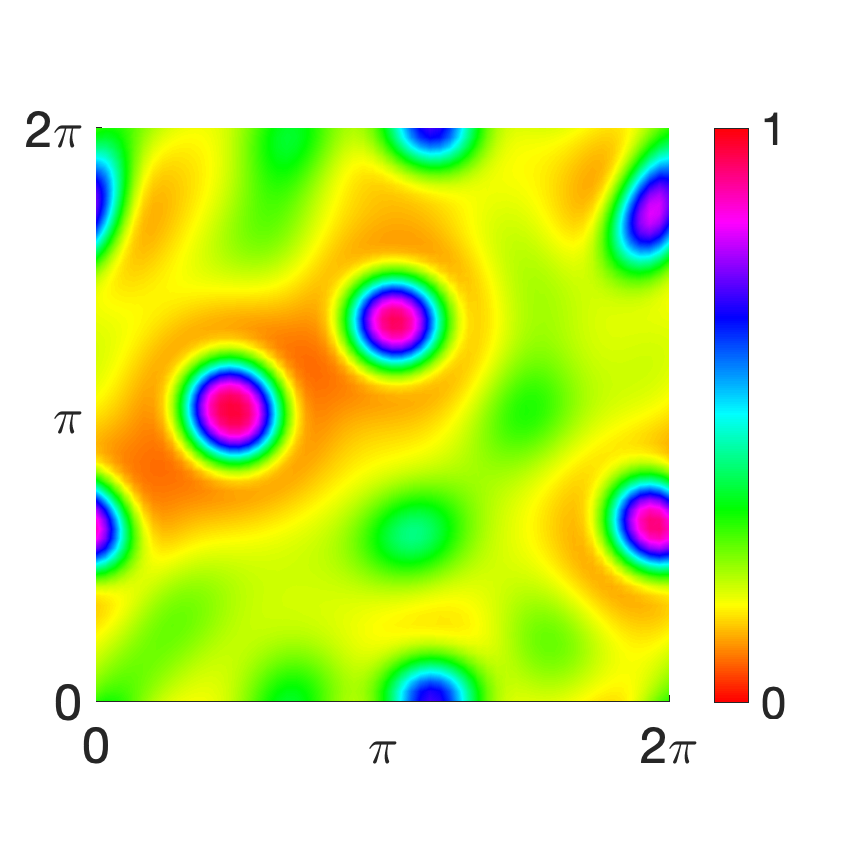}\\[-5mm]
	\includegraphics[width=52mm]{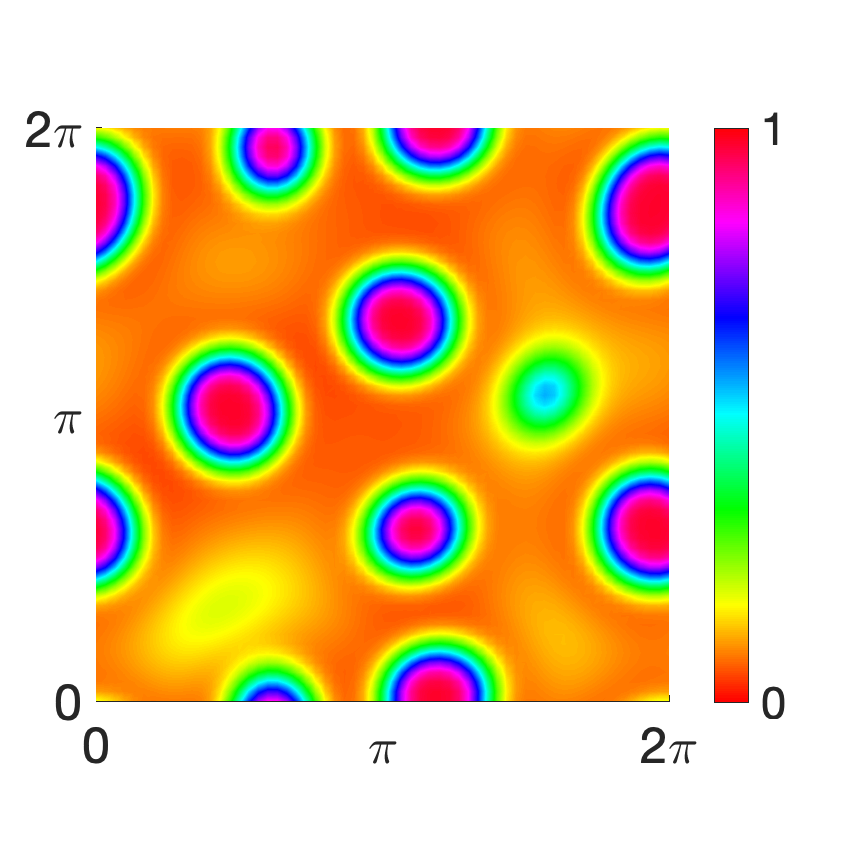}
	\includegraphics[width=52mm]{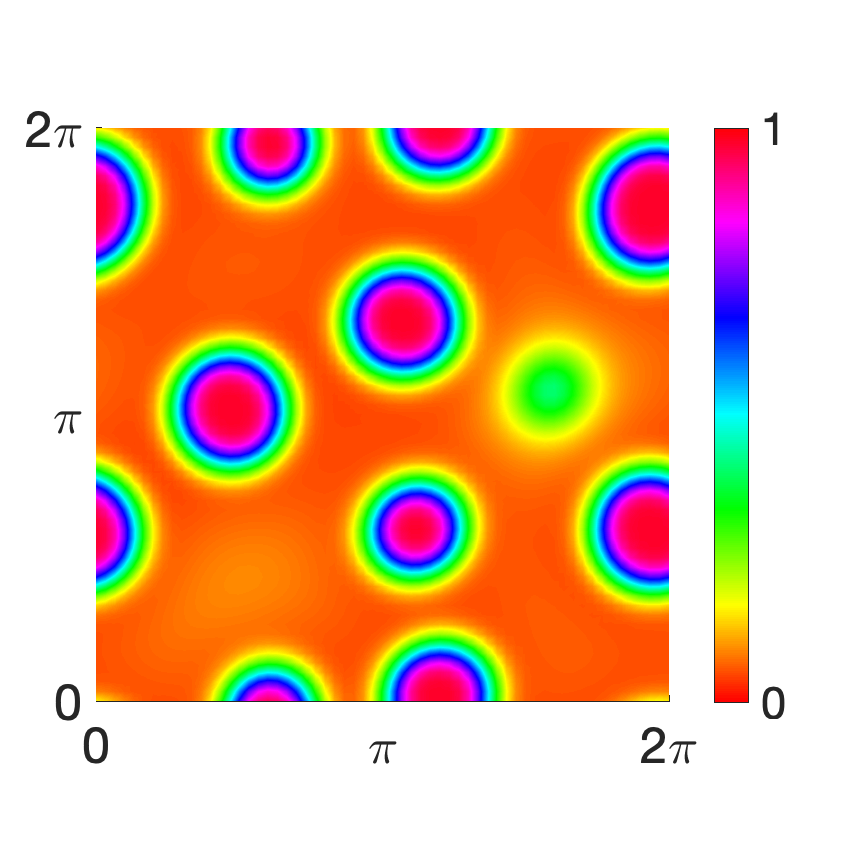}
	\includegraphics[width=52mm]{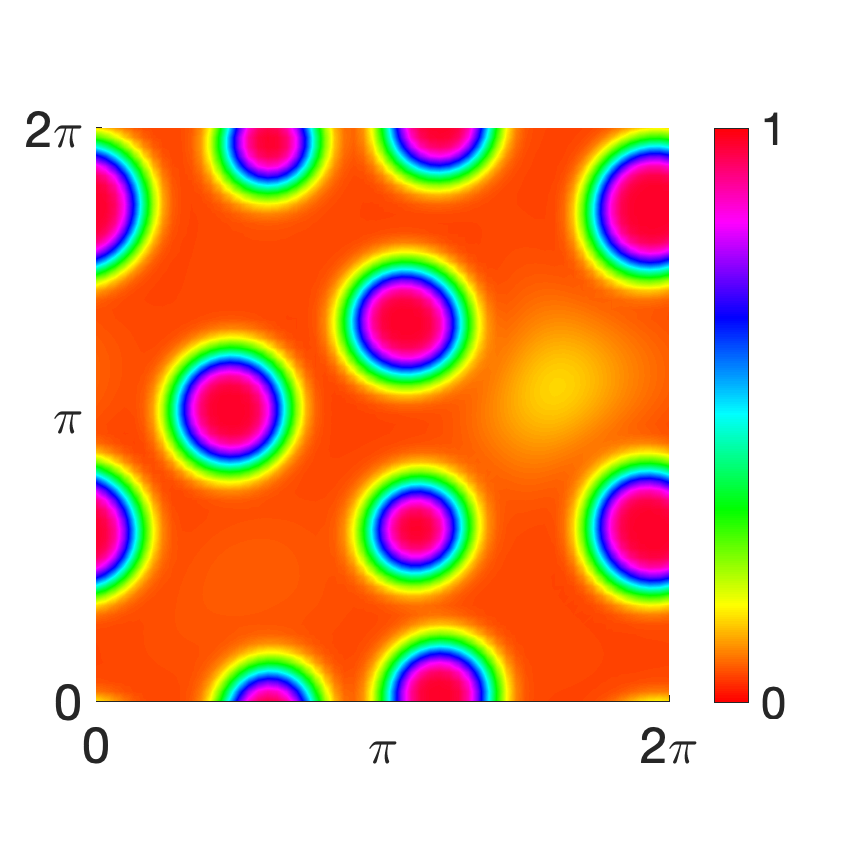}
	\caption{Fiber phase fraction $\phi$ at various times $t = 0.2,\, 2,\, 4,\,6,\,8,\,10$ with $\eps=0.3$.}
	\label{fig.phi.2}
\end{figure}
	
\begin{figure}[ht]
	\includegraphics[width=52mm]{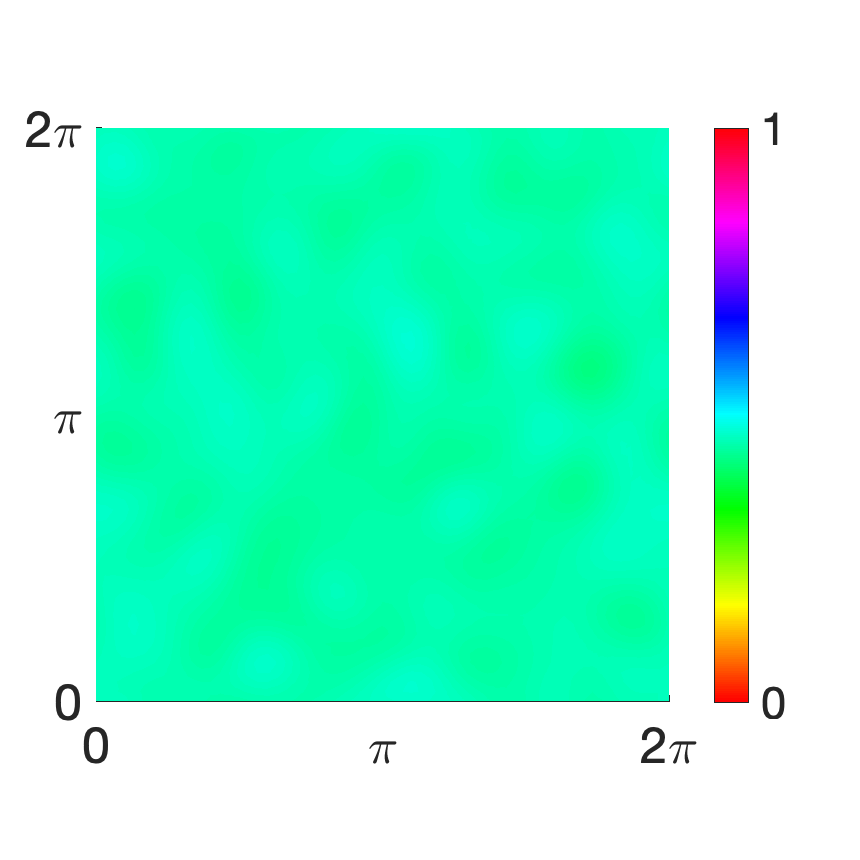}
	\includegraphics[width=52mm]{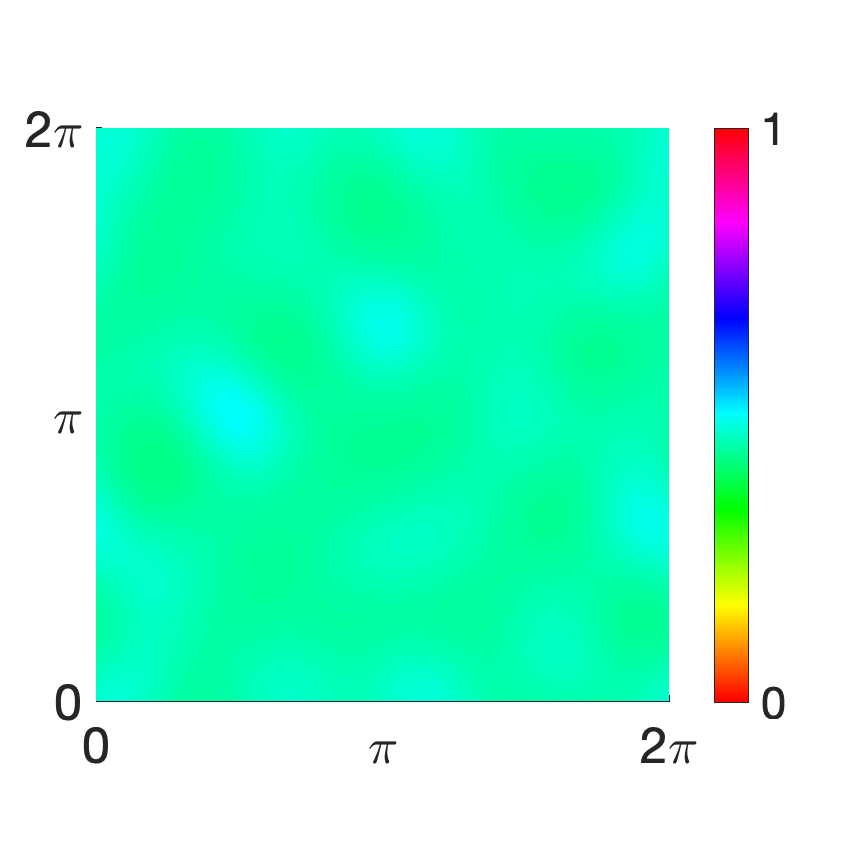}
	\includegraphics[width=52mm]{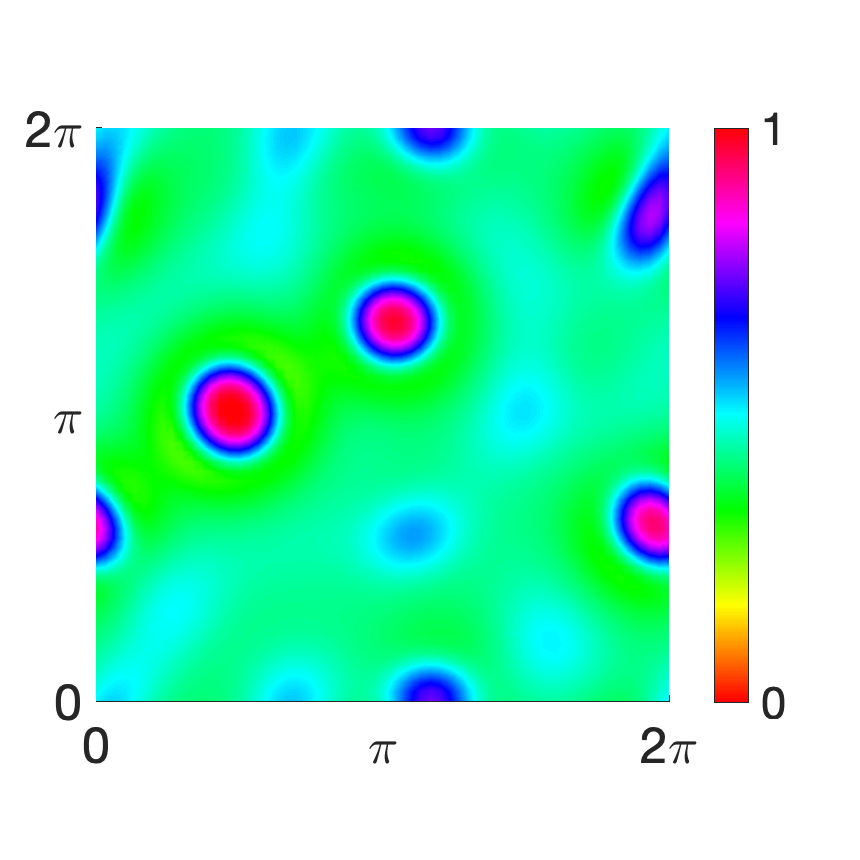}\\[-5mm]
	\includegraphics[width=52mm]{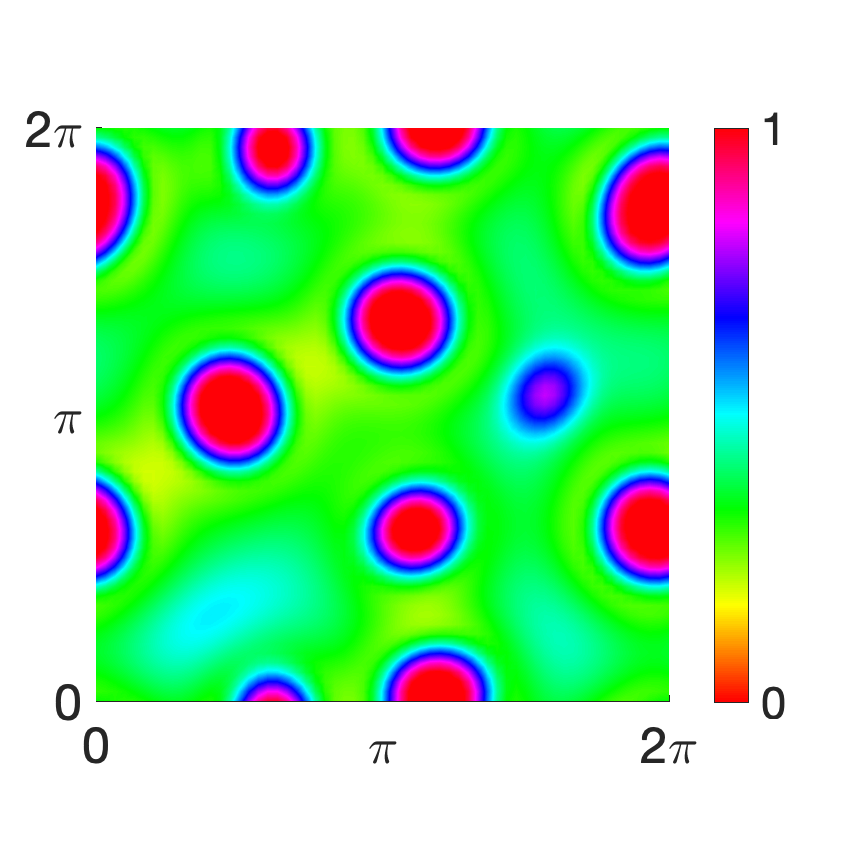}
	\includegraphics[width=52mm]{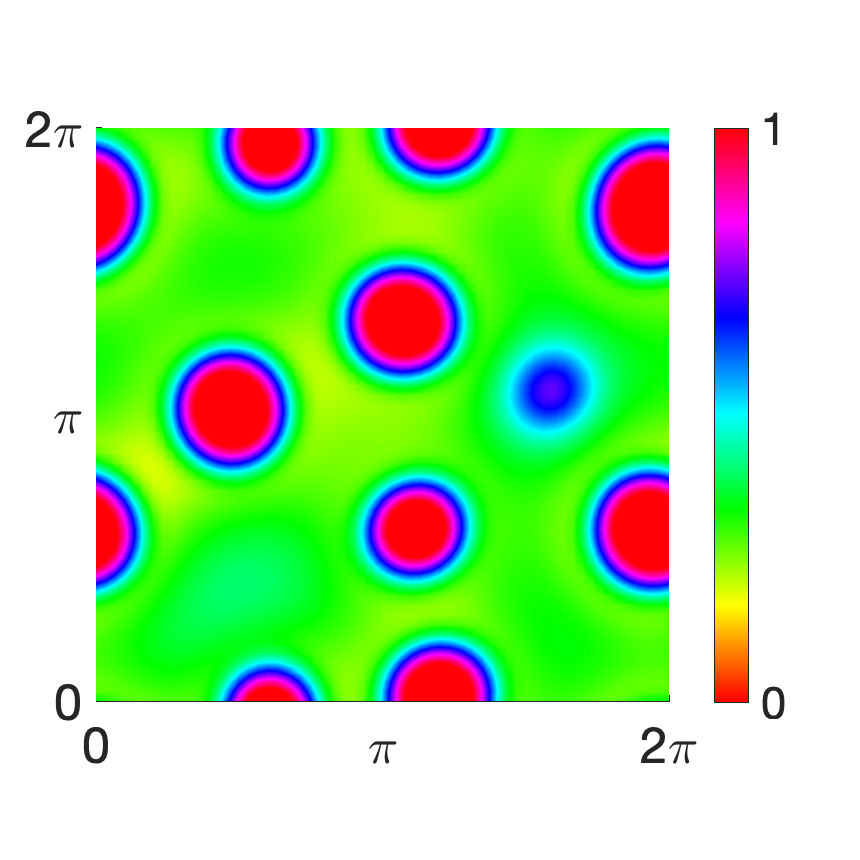}
	\includegraphics[width=52mm]{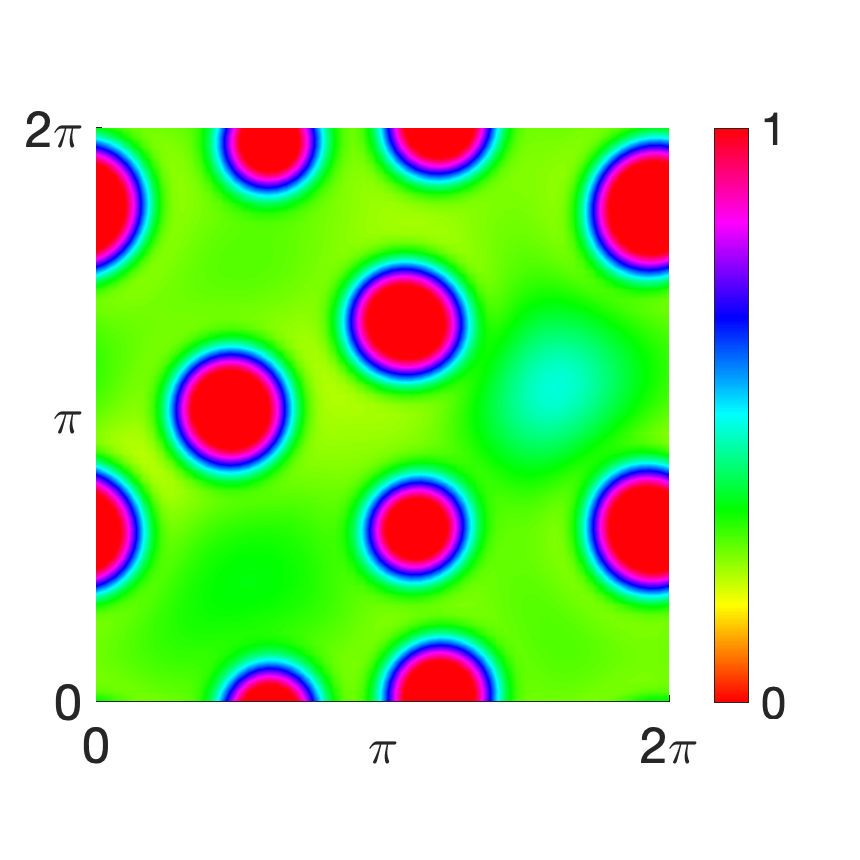}
	\caption{Nutrient concentration $c$ at various times $t = 0.2,\, 2,\, 4,\,6,\,8,\,10$ with $\eps=0.3$.}
	\label{fig.c.2}
\end{figure}


\section{Conclusion}

{In this paper, a Cahn--Hilliard cross-diffusion system modeling the pre-patterning of lymphatic vessel morphology is considered, and a fully discrete finite-element method based on the idea of convex splitting is proposed. The existence of discrete solutions as well as structure-preserving properties of the scheme are proved, including energy stability, conservation of the solute mass, and preservation of lower and upper bounds of the fiber phase fraction. Numerical experiments performed in FreeFEM have confirmed the analytical results. We have verified the boundedness of the discrete concentration numerically. The scheme presents a robust approach for simulating complex biological processes involving cross-diffusion phenomena.} 

{As our scheme is of first order in time only, one may discuss the extension to higher-order approximations. A second-order BDF scheme has been proposed for a finite-difference approximation of the Cahn--Hillard equation. The difficulty is to extend this scheme to cross-diffusion systems. General one-leg multistep semi-discrete schemes have been presented in \cite{JuMi15}, and second-order convergence results have been shown. However, the proof only works for non-logarithmic (power-law type) energy functions and cannot be applied to Boltzmann-type entropies. The error analysis for higher-order temporal numerical schemes for cross-diffusion systems is delicate and a challenging work in the future.}


\end{document}